\documentclass[10pt,oneside]{amsart}

\usepackage[
paper=a4paper,
headsep=15pt,text={135mm,214mm},centering
]{geometry}

\usepackage[bookmarks]{hyperref}
\usepackage{amssymb,amsxtra}
\usepackage{graphicx,xcolor}
\usepackage{float,array,calc,booktabs,stackrel}
\usepackage{enumitem}\setlist[enumerate,1]{font=\upshape}

\usepackage{tikz}
\usetikzlibrary{cd,calc,positioning,arrows,decorations.markings}

\definecolor{todo-background-color}{gray}{0.95}
\usepackage[
  textwidth=1.1in,
  backgroundcolor=todo-background-color,
  bordercolor=black,
  linecolor=black,
  textsize=footnotesize
]{todonotes}

    \iftrue
    
    \makeatletter
    \def\@settitle{%
      \vspace*{-10pt}
      \begin{flushleft}%
        \LARGE\bfseries
        \strut\@title\strut
      \end{flushleft}%
    }
    \def\@setauthors{%
      \begingroup
      \def\thanks{\protect\thanks@warning}%
      \trivlist
      \raggedright
      \large \@topsep27\p@\relax
      \advance\@topsep by -\baselineskip
    \item\relax
      \author@andify\authors
      \def\\{\protect\linebreak}%
      \authors
      \ifx\@empty\contribs
      \else
      ,\penalty-3 \space \@setcontribs
      \@closetoccontribs
      \fi
      \normalfont
      \endtrivlist
      \endgroup
    }
    \def\@setaddresses{\par
      \nobreak \begingroup
      \small\raggedright
      \def\author##1{\nobreak\addvspace\smallskipamount}%
      \def\\{\unskip, \ignorespaces}%
      \interlinepenalty\@M
      \def\address##1##2{\begingroup
        \par\addvspace\bigskipamount\noindent
        \@ifnotempty{##1}{(\ignorespaces##1\unskip) }%
        {\ignorespaces##2}\par\endgroup}%
      \def\curraddr##1##2{\begingroup
        \@ifnotempty{##2}{\nobreak\noindent\curraddrname
          \@ifnotempty{##1}{, \ignorespaces##1\unskip}\/:\space
          ##2\par}\endgroup}%
      \def\email##1##2{\begingroup
        \@ifnotempty{##2}{\nobreak\noindent E-mail address%
          \@ifnotempty{##1}{, \ignorespaces##1\unskip}\/:\space
          \ttfamily##2\par}\endgroup}%
      \def\urladdr##1##2{\begingroup
        \def~{\char`\~}%
        \@ifnotempty{##2}{\nobreak\noindent\urladdrname
          \@ifnotempty{##1}{, \ignorespaces##1\unskip}\/:\space
          \ttfamily##2\par}\endgroup}%
      \addresses
      \endgroup
      \global\let\addresses=\@empty
    }
    \def\@setabstracta{%
      \ifvoid\abstractbox
      \else
      \skip@17pt \advance\skip@-\lastskip
      \advance\skip@-\baselineskip \vskip\skip@
      \box\abstractbox
      \prevdepth\z@ 
      \vskip-28pt
      \fi
    }
    \renewenvironment{abstract}{%
      \ifx\maketitle\relax
      \ClassWarning{\@classname}{Abstract should precede
        \protect\maketitle\space in AMS document classes; reported}%
      \fi
      \global\setbox\abstractbox=\vtop \bgroup
      \normalfont\small
      \list{}{\labelwidth\z@
        \leftmargin0pc \rightmargin\leftmargin
        \listparindent\normalparindent \itemindent\z@
        \parsep\z@ \@plus\p@
        
      }%
    \item[\hskip\labelsep\bfseries\abstractname.]%
    }{%
      \endlist\egroup
      \ifx\@setabstract\relax \@setabstracta \fi
    }

    \def\ps@headings{\ps@empty
      \def\@evenhead{%
        \setTrue{runhead}%
        \normalfont\scriptsize
        \rlap{\thepage}\hfill
        \def\thanks{\protect\thanks@warning}%
        \leftmark{}{}}%
      \def\@oddhead{%
        \setTrue{runhead}%
        \normalfont\scriptsize
        \def\thanks{\protect\thanks@warning}%
        \rightmark{}{}\hfill \llap{\thepage}}%
      \let\@mkboth\markboth
    }\ps@headings

    \def\section{\@startsection{section}{1}%
      \z@{-1.4\linespacing\@plus-.5\linespacing}{.8\linespacing}%
      {\normalfont\bfseries\Large}}
    \def\subsection{\@startsection{subsection}{2}%
      \z@{-.8\linespacing\@plus-.3\linespacing}{.5\linespacing\@plus.2\linespacing}%
      {\normalfont\bfseries\large}}
    \def\subsubsection{\@startsection{subsubsection}{3}%
      \z@{.7\linespacing\@plus.2\linespacing}{-1.5ex}%
      {\normalfont\itshape}}
    \def\paragraph{\@startsection{paragraph}{4}%
      \z@{.7\linespacing\@plus.2\linespacing}{-1.5ex}%
      {\normalfont\itshape}}
    \def\@secnumfont{\bfseries}

    \renewcommand\contentsnamefont{\bfseries}
    \def\@starttoc#1#2{\begingroup
      \setTrue{#1}%
      \par\removelastskip\vskip\z@skip
      \@startsection{}\@M\z@{\linespacing\@plus\linespacing}%
      {.5\linespacing}{
        \contentsnamefont}{#2}%
      \ifx\contentsname#2%
      \else \addcontentsline{toc}{section}{#2}\fi
      \makeatletter
      \@input{\jobname.#1}%
      \if@filesw
      \@xp\newwrite\csname tf@#1\endcsname
      \immediate\@xp\openout\csname tf@#1\endcsname \jobname.#1\relax
      \fi
      \global\@nobreakfalse \endgroup
      \addvspace{32\p@\@plus14\p@}%
      \let\tableofcontents\relax
    }
    \def\contentsname{Contents}
    \def\l@section{\@tocline{2}{.5ex}{0mm}{5pc}{}}
    \def\l@subsection{\@tocline{2}{0pt}{2em}{5pc}{}}
    \makeatother

    \fi

\def\to{\mathchoice{\longrightarrow}{\rightarrow}{\rightarrow}{\rightarrow}}
\makeatletter
\newcommand{\shortxra}[2][]{\ext@arrow 0359\rightarrowfill@{#1}{#2}}
\def\longrightarrowfill@{\arrowfill@\relbar\relbar\longrightarrow}
\newcommand{\longxra}[2][]{\ext@arrow 0359\longrightarrowfill@{#1}{#2}}
\renewcommand{\xrightarrow}[2][]{\mathchoice{\longxra[#1]{#2}}%
  {\shortxra[#1]{#2}}{\shortxra[#1]{#2}}{\shortxra[#1]{#2}}}
\makeatother

\makeatletter
\def\addtagsub#1{\let\oldtf=\tagform@\def\tagform@##1{\oldtf{##1}\hbox{$_{#1}$}}}
\makeatother

\makeatletter
\def\Nopagebreak{\@nobreaktrue\nopagebreak}
\makeatother

\makeatletter
\newtheoremstyle{theorem-giventitle}
        {}{}              
        {\itshape}                      
        {}                              
        {\bfseries}                     
        {.}                             
        {\thm@headsep}                             
        {\thmnote{\bfseries#3}}
\newtheoremstyle{theorem-givenlabel}
        {}{}              
        {\itshape}                      
        {}                              
        {\bfseries}                     
        {.}                             
        {\thm@headsep}                             
        {\thmname{#1}~\thmnumber{#3}\setcurrentlabel{#3}}
\newtheoremstyle{definition-giventitle}
        {}{}              
        {}                      
        {}                              
        {\bfseries}                     
        {.}                             
        {\thm@headsep}                             
        {\thmnote{\bfseries#3}}
\def\setcurrentlabel#1{\gdef\@currentlabel{#1}}
\makeatother

\newtheorem{theorem}{Theorem}[section]
\newtheorem{theoremalpha}{Theorem}

\newtheorem{corollary}[theorem]{Corollary}
\newtheorem{lemma}[theorem]{Lemma}

\theoremstyle{definition}
\newtheorem{definition}[theorem]{Definition}

\newtheorem{remark}[theorem]{Remark}

\theoremstyle{theorem-giventitle}
\newtheorem{theorem-named}{}
\theoremstyle{theorem-givenlabel}
\newtheorem{theorem-labeled}{Theorem}

\theoremstyle{definition-giventitle}
\newtheorem{definition-named}{}
\newtheorem{conjecture-named}{}
\newtheorem{case-named}{}

\numberwithin{equation}{section}

\def\Z{\mathbb{Z}}

\def\Q{\mathbb{Q}}
\def\W{\mathbb{W}}
\def\cM{\mathcal{M}}
\def\cR{\mathcal{R}}
\def\cT{\mathcal{T}}
\def\sB{\mathsf{B}}
\def\sD{\mathsf{D}}
\def\sK{\mathsf{K}}
\def\sL{\mathsf{L}}
\def\sW{\mathsf{W}}
\def\sm{\smallsetminus}
\DeclareMathOperator\Ker{Ker}

\DeclareMathOperator\inte{int}

\DeclareMathOperator\Arf{Arf}
\DeclareMathOperator\SLev{SL}
\DeclareMathOperator\slev{sl}

\def\otimesover#1{\mathbin{\mathop{\otimes}_{#1}}}
\def\cupover#1{\mathbin{\mathop{\cup}_{#1}}}
\def\csum{\mathbin{\#}}
\def\csumover#1{\mathbin{\mathop{\csum}_{#1}}}
\def\ol#1{\overline{#1}{}}

\makeatletter
\def\rddots{\mathinner{\mkern1mu\raise\p@
    \vbox{\kern7\p@\hbox{.}}\mkern2mu
    \raise4\p@\hbox{.}\mkern2mu\raise7\p@\hbox{.}\mkern1mu}}
\makeatother
\def\smallItree#1{
  \vcenter{\hbox{\tikz[x=1pt,y=1pt,thick,fill,inner sep=0]{
        \draw (0,0) circle(.6pt)--(0:10) node[right,xshift=.5ex]{$#1$}}}}}
\def\smalltwItree#1{
  \vcenter{\hbox{\tikz[x=1pt,y=1pt,thick,fill,inner sep=0]{
        \draw (0,0) node[left,xshift=-.5ex]{$\tw$} --(0:10) node[right,xshift=.5ex]{$#1$}}}}}
\def\smallYtree#1#2#3{%
  \vcenter{\hbox{\tikz[x=1pt,y=1pt,inner sep=0,thick,fill]{
        \draw (0,0)--++(180:7) node[left,xshift=-.5ex]{$#1$} (0,0)
        -- +(30:8) node[right,xshift=.5ex]{$\scriptstyle #3$}
        (0,0) -- +(-30:8)
        node[right,xshift=.5ex]{$\scriptstyle #2$}}}}}
\def\smallrootedYtree#1#2{%
  \vcenter{\hbox{\tikz[x=1pt,y=1pt,inner sep=0,thick,fill]{
        \draw (0,0)--++(180:7) circle(.6pt) (0,0)
        -- +(30:8) node[right,xshift=.5ex]{$\scriptstyle #2$}
        (0,0) -- +(-30:8)
        node[right,xshift=.5ex]{$\scriptstyle #1$}}}}}
\def\drawtw#1#2{\tikz[baseline=-.07ex,cap=round,scale=#1,line width=#2]
  {\draw (0,1ex) ..controls+(-.5ex,0) and +(-.5ex,0)..
    (0ex,0ex) ..controls+(.5ex,0) and +(-.5ex,0)..
    (.85ex,1ex) ..controls+(.5ex,0) and +(.5ex,0).. (.85ex,0);}}
\newsavebox{\ttwbox}\newsavebox{\stwbox}\newsavebox{\sstwbox}
\sbox{\ttwbox}{\drawtw{.9}{.6pt}}
\sbox{\stwbox}{\drawtw{.54}{.5pt}}
\sbox{\sstwbox}{\drawtw{.45}{.45pt}}
\def\tw{{\mathchoice{\usebox{\ttwbox}}{\usebox{\ttwbox}}{\usebox{\stwbox}}{\usebox{\sstwbox}}}}

\begin{document}

\title
[Rational Whitney tower filtration of links]
{Rational Whitney tower filtration of links}

\author{Jae Choon Cha}
\address{
  Department of Mathematics\\
  POSTECH\\
  Pohang Gyeongbuk 37673\\
  Republic of Korea\quad
  \linebreak
  School of Mathematics\\
  Korea Institute for Advanced Study \\
  Seoul 02455\\
  Republic of Korea
}
\email{jccha@postech.ac.kr}

\def\subjclassname{\textup{2010} Mathematics Subject Classification}
\expandafter\let\csname subjclassname@1991\endcsname=\subjclassname
\expandafter\let\csname subjclassname@2000\endcsname=\subjclassname

\begin{abstract}
  We present complete classifications of links in the 3-sphere modulo
  framed and twisted Whitney towers in a rational homology 4-ball.
  This provides a geometric characterization of the vanishing of the
  Milnor invariants of links in terms of Whitney towers.  Our result
  also says that the higher order Arf invariants, which are
  conjectured to be nontrivial, measure the potential difference
  between the Whitney tower theory in rational homology 4-balls and
  that in the 4-ball extensively developed by Conant, Schneiderman and
  Teichner.
\end{abstract}

\maketitle


\section{Introduction}

Topology of dimension 4 is different from high dimensions because the
Whitney move may fail.  The essential problem is to find an embedded
Whitney disk along which a pair of intersections of two sheets could
be removed by a Whitney move.  Once an immersed Whitney disk is
obtained from fundamental group data, one may try to remove double
points of the disk by finding a next stage of Whitney disks.
Iterating this, we are led to the notion of a \emph{Whitney tower}.

Since work of Cochran, Orr and
Teichner~\cite{Cochran-Orr-Teichner:1999-1}, concordance of knots and
links, which is the ``local case'' of general disk embedding, has been
extensively studied via frameworks formulated in terms of Whitney
towers.  In this paper, we will focus on asymmetric Whitney towers in
dimension 4 bounded by links in $S^3$, motivated from work of Conant,
Schneiderman and Teichner~\cite{Conant-Schneiderman-Teichner:2011-1,
  Conant-Schneiderman-Teichner:2012-2,
  Conant-Schneiderman-Teichner:2012-3,
  Conant-Schneiderman-Teichner:2012-4,
  Conant-Schneiderman-Teichner:2012-1}.  Whitney towers come in two
flavors: \emph{framed} and \emph{twisted}.  Whitney towers we consider
have an \emph{order}, which is a nonnegative integer measuring the
number of iterated stages.  Precise definitions can be found in
Section~\ref{section:whitney-towers}.

The main result of this paper is a complete classification of links in
$S^3$ modulo Whitney towers in \emph{rational homology 4-balls}.  To
state our result, we use the following notation.  Fix $m>0$, and let
$\ol{\W}_n^\tw$ be the set of $m$-component links in $S^3$ bounding a
twisted Whitney tower of order $n$ in a rational homology 4-ball with
boundary~$S^3$.  We define the graded quotient $\ol{\sW}_n^\tw$ of
$\ol{\W}_n^\tw$ by the condition that $L$ and $L'$ in $\ol{\W}_n^\tw$
represent the same element in $\ol{\sW}_n^\tw$ if and only if a band
sum of $L$ and $-L'$ lies in~$\ol\W^\tw_{n+1}$.  In fact, in
Section~\ref{subsection:twisted-graded-quotients}, we will show that
it is an equivalence relation, and $L\in \ol\W^\tw_{n+1}$ if and only
if $[L]=0$ in~$\ol\sW^\tw_n$.  So we may write
$\ol\sW^\tw_n = \ol\W^\tw_n/\ol\W^\tw_{n+1}$.


\begin{theoremalpha}
  \label{theorem:classification-intro}
  \leavevmode\Nopagebreak
  \begin{enumerate}
  \item Band sum is a well-defined operation on the set
    $\ol\sW^\tw_{n}$, independent of the choice of bands, and
    $\ol\sW^\tw_{n}$ is an abelian group under band sum.
  \item $\ol{\sW}^\tw_{n}$ is classified by the Milnor invariants of
    order $n$ \textup{(}$=$ length $n+2$\textup{)}.
  \item $\ol{\sW}^\tw_{n}$ is a free abelian group of rank
    $m \cR(m,n+1)-\cR(m,n+2)$, where
    $\cR(m,n) = \frac 1n \sum_{d\mid n} \phi(d)\cdot m^{n/d}$ and
    $\phi(d)$ is the M\"obius function.
  \end{enumerate}
\end{theoremalpha}

We remark that $\cR(m,n)$ is the rank of the degree $n$ part of the
free Lie algebra on $m$ variables, due to Witt (e.g.,
see~\cite[Section~5.6]{Magnus-Karrass-Solitar:1966-1}), and
$m\cR(m,n+1)-\cR(m,n+2)$ is the number of linearly independent Milnor
invariants of order $n$, due to Orr~\cite{Orr:1989-1}.  The proof of
Theorem~\ref{theorem:classification-intro} is given in
Section~\ref{subsection:twisted-graded-quotients}.  Especially see
Theorem~\ref{theorem:computation-of-twisted-graded-quotient}.

We also present a complete classification of links modulo
\emph{framed} Whitney towers.  Briefly speaking, we define the framed
analog $\ol\W_n$ and its graded quotient $\ol\sW_n$ along the same
lines using framed Whitney towers in rational homology 4-balls instead
of twisted Whitney towers, so that $L\in \ol\W_{n+1}$ if and only if
$[L]=0$ in~$\ol\sW_n$.  We prove that $\ol{\sW}_n$ is an abelian
group, and $\ol{\sW}_n$ is completely classified by the Milnor
invariants and the higher order Sato-Levine invariants introduced
in~\cite{Conant-Schneiderman-Teichner:2012-2}.  It turns out that
$\ol{\sW}_n$ is isomorphic to the direct sum of a certain determined
number of copies of $\Z$ and~$\Z_2$.  Details are given in
Section~\ref{section:framed-classification}.  In particular see
Theorem~\ref{theorem:structure-graded-quotient-framed}.  We remark
that even the proof that $\ol{\sW}_n$ is an abelian group under band
sum is not straightforward.

The above results remain true when we replace $\Q$ by any subring of
$\Q$ in which $2$ is invertible.

\begin{theoremalpha}
  \label{theorem:general-coefficients-intro}
  For any subring $R$ of $\Q$ containing $\tfrac12$, a link in $S^3$
  bounds a twisted Whitney tower of order $n$ in an $R$-homology
  4-ball if and only if the link bounds a twisted Whitney tower of
  order $n$ in a rational homology 4-ball.  The framed case analog
  holds too.
\end{theoremalpha}

We prove the twisted case of
Theorem~\ref{theorem:general-coefficients-intro} in
Section~\ref{subsection:rational-whitney-filtration}.  In particular
see Theorem~\ref{theorem:rational-triviality-characterization}.  For
the framed case, see
Theorem~\ref{theorem:rational-triviality-characterization-framed} in
Section~\ref{subsection:rational-whitney-filtration-framed}.

\subsubsection*{Milnor invariants and rational Whitney towers}

The problem of understanding the Milnor invariants geometrically has
been addressed by numerous authors.  Especially Igusa and Orr proved
the \emph{$k$-slice conjecture}, which asserts that a link $L$ has
vanishing Milnor invariants of length $\le 2k$ if and only if $L$
bounds disjoint surfaces in $D^4$ such that each loop on these
surfaces can be pushed off to a loop lying in the $k$th lower central
subgroup of the fundamental group of the complement of the
surfaces~\cite{Igusa-Orr:2001-1}.  A significantly strengthened
version of the Igusa-Orr theorem was given in
\cite[Theorem~18]{Conant-Schneiderman-Teichner:2012-3} by Conant,
Schneiderman and Teichner.

As a consequence of our main result, we present a geometric
characterization of the vanishing of the Milnor invariants in terms of
Whitney towers.

\begin{theoremalpha}
  \label{theorem:geometric-characterization-milnor-intro}
  A link $L$ in $S^3$ has vanishing Milnor invariants of order $\le n$
  \textup{(}or equivalently length $\le n+2$\textup{)} if and only if
  $L$ bounds a twisted Whitney tower of order $n+1$ in a rational
  homology 4-ball.
\end{theoremalpha}

We remark that $L$ bounds a twisted Whitney tower of order $n+1$ in a
rational homology 4-ball if and only if $L$ bounds a twisted capped
grope of class $n+2$ in a rational homology 4-ball, due to
\cite[Theorem~5]{Schneiderman:2006-1}
and~\cite[Lemma~23]{Conant-Schneiderman-Teichner:2012-3}.  We prove
Theorem~\ref{theorem:geometric-characterization-milnor-intro} in
Section~\ref{subsection:rational-whitney-filtration} as a part of
Theorem~\ref{theorem:rational-triviality-characterization}.

\subsubsection*{Higher order Arf invariants and rational Whitney towers}

In their study of link concordance via Whitney towers in~$D^4$,
Conant, Schneiderman and Teichner introduced the \emph{higher order
  Arf invariant}~$\Arf_k$ ($k\ge 1$).  Together with the Milnor
invariants, $\Arf_k$ forms a complete set of invariants used to
present classifications of links modulo twisted Whitney towers
in~$D^4$.  Understanding the higher order Arf invariants, which remain
mysterious yet, is the most significant open problem in the study of
finite asymmetric Whitney towers.  In particular the \emph{higher
  order Arf invariant conjecture} asserts that $\Arf_k$ are
nontrivial~\cite[Conjecture~1.17]{Conant-Schneiderman-Teichner:2012-2}.

Our main result provides a geometric interpretation of the
\mbox{(non-)vanishing} of the higher order Arf invariants.  Briefly,
the higher order Arf invariants measure the difference between a
bounding Whitney tower in the standard 4-ball and one in a rational
homology 4-ball.


\begin{theoremalpha}
  \label{theorem:higher-order-arf-rational-tower-intro}
  For each $n\ge 0$, the following statements are equivalent.
  \Nopagebreak
  \begin{enumerate}
  \item $\Arf_k \equiv 0$ for $4k-2\le n$.
  \item A link $L\subset S^3$ bounds a twisted Whitney tower of order
    $n+1$ in $D^4$ if and only if $L$ bounds a twisted Whitney tower
    of order $n+1$ in a rational homology 4-ball.
  \end{enumerate}
\end{theoremalpha}

We prove Theorem~\ref{theorem:higher-order-arf-rational-tower-intro}
at the end of Section~\ref{subsection:twisted-graded-quotients}.
Especially see
Corollary~\ref{corollary:higher-order-arf-graded-quotient}.


\subsubsection*{Some remarks on our approach}

The proofs of our main results hinge, in an essential way, on the work
of Conant, Schneiderman and Teichner on Whitney towers in $D^4$
\cite{Conant-Schneiderman-Teichner:2012-2,
  Conant-Schneiderman-Teichner:2012-3,
  Conant-Schneiderman-Teichner:2012-1,
  Conant-Schneiderman-Teichner:2012-4} which is summarized
in~\cite{Conant-Schneiderman-Teichner:2011-1}.

They formulate algebraic analogs of the geometric theory of Whitney
towers, in terms of intersection data of Whitney disks, and present
complete classifications of the algebraic side using their proof of a
conjecture of Levine~\cite{Levine:2001-1,Levine:2001-2}.  To relate
this to the geometric side, they prove a key result called the
\emph{order raising theorem}~\cite[Theorems 1.9, 2.6, 2.10
and~4.4]{Conant-Schneiderman-Teichner:2012-2}, whose origin goes back
to~\cite[Theorem~2]{Schneiderman-Teichner:2004-1}.  It essentially
says that the vanishing of algebraic intersection data is sufficient
to raise the order of a Whitney tower in~$D^4$.  This approach gives
Whitney tower concordance classifications of links, modulo
indeterminacy from a certain not-yet-understood part of the
correspondence between the algebraic and geometric sides, which the
higher order Arf invariant conjecture concerns.

A natural attempt for the study of Whitney towers in rational homology
4-balls, or more generally in general 4-manifolds, would be to develop
a non-simply-connected version of the above algebraic theory and order
rasing theorem.  This appears to be a very interesting problem, whose
solution seems far from being straightforward.

Instead, we present a different approach.  We identify exactly which
part of the Conant-Schneiderman-Teichner theory of Whitney towers in
$D^4$ is annihilated in rational homology 4-balls.  In fact we show
that the information from the Milnor invariants (and the higher order
Sato-Levine invariants in the framed odd order case) survives, while
the higher order Arf invariants are eliminated when passed to the
rational theory, as indicated in
Theorem~\ref{theorem:higher-order-arf-rational-tower-intro}\@.
Put differently, the part not yet fully understood in the integral
theory is exactly the information annihilated in the rational theory.
This leads us to rational Whitney tower classification results
\emph{without indeterminacy}, as stated in
Theorem~\ref{theorem:classification-intro} and
Theorem~\ref{theorem:structure-graded-quotient-framed}.

To show that the Milnor invariant information is preserved in the
rational theory, we first show a Milnor type theorem for Whitney
towers in a rational homology 4-ball, which computes the lower central
series quotients of the complement fundamental group.  See
Theorem~\ref{theorem:milnor-theorem-rational-tower}.  Using this and
commutator calculus on a Whitney tower, we show that Milnor invariants
(and higher order Sato-Levine invariants) are determined by a Whitney
tower in a rational homology 4-ball.  See
Theorem~\ref{theorem:whitney-tower-milnor-invariant}.  This
generalizes an earlier result
in~\cite{Conant-Schneiderman-Teichner:2012-3}.

The elimination of the higher order Arf invariants in the rational
theory generalizes an earlier result too.  Indeed, the figure eight
knot, which has nontrivial Arf invariant, is known to bound a slice
disk in a rational homology 4-ball~\cite{Cha:2003-1}, and this tells
us that the classical Arf invariant is not preserved under rational
concordance.  Our generalization to the higher order case is based on
this.  Precise formulations and proofs are given in
Lemmas~\ref{lemma:kernel-realization-by-rationally-slice-links}
and~\ref{lemma:kernel-realization-by-rationally-slice-links-framed}.

\subsubsection*{Organization of the paper}

In Section~\ref{section:whitney-towers}, we review the definitions of
Whitney towers and trees representing intersection data of Whitney
disks.  In Section~\ref{section:milnor-invariant-rational-tower}, we
investigate the relationship of the Milnor invariants of links and
bounding Whitney towers in a rational homology 4-ball.  In
Section~\ref{section:rational-theory-to-integral-theory}, we study
twisted Whitney towers in a rational homology 4-ball.  We give a
complete characterization of links bounding a twisted Whitney tower of
a given order and prove Theorem~\ref{theorem:classification-intro}\@.
Section~\ref{section:framed-classification} is devoted to the study of
framed Whitney towers in a rational homology 4-ball.

\subsubsection*{Acknowledgement}

The author thanks an anonymous referee for careful comments.  This
work was partially supported by NRF grant 2013067043.

\section{Whitney towers and associated trees}
\label{section:whitney-towers}

In this section we will review definitions of twisted and framed
asymmetric Whitney towers in 4-manifolds, and discuss uni-trivalent
trees which arise naturally in the study of iterated intersections of
surfaces, particularly for Whitney towers (e.g.,
see~\cite{Cochran:1990-1, Conant-Teichner:2004-1,
  Conant-Teichner:2004-2, Schneiderman:2006-1,
  Conant-Schneiderman-Teichner:2007-1,
  Conant-Schneiderman-Teichner:2012-2,
  Conant-Schneiderman-Teichner:2012-3}).  Readers who are familiar
with them may skip to
Section~\ref{section:milnor-invariant-rational-tower}, after reading
this paragraph.  In this paper a Whitney tower is always assumed to be
union-of-disks-like (defined below), except the case of a Whitney
tower concordance, which is union-of-annuli-like.  Manifolds and
immersed surfaces are always oriented.

\subsection{Definitions of Whitney towers}
\label{subsection:definition-whitney-tower}

In what follows, a \emph{sheet} is an open subset of an immersed
surface in a 4-manifold.

\begin{definition}[Twisted and framed Whitney disk]
  \label{definition:twisted-framed-whitney-disk}
  Suppose $X$ is a 4-manifold and $p$, $q$ are two intersections of
  opposite signs of two connected sheets $A$ and $B$ in~$X$.  A
  \emph{Whitney circle} pairing $p$ and $q$ is an embedded circle
  $\alpha$ which is the union of an arc on $A$ joining $p$ and $q$ and
  another arc on $B$ joining $p$ and~$q$.  A \emph{Whitney disk}
  pairing $p$ and $q$ is an immersed disk $D$ in $X$ bounded by a
  Whitney circle $\alpha$.  We require that there is a collar
  neighborhood of $\partial D$ in $D$ whose intersection with
  $A\cup B$ is~$\partial D$, while the complement of the collar is
  allowed to intersect the sheets.

  For an immersed disk $D$, we call the restriction of the unique
  framing of $D$ on $\partial D$ the \emph{disk framing}.  On the
  boundary of a Whitney disk $D$, the tangential direction of one of
  the involved sheets and the common normal direction of $D$ and the
  other sheet defines a framing, which we call the \emph{Whitney
    framing}.  Using $\operatorname{SO}(2)=\Z$, the disk framing with
  respect to the Whitney framing determines an integer $\omega(D)$
  called the \emph{twisting number} of~$D$.  If $\omega(D)=0$, then
  $D$ is called \emph{framed}.  When we do not require a disk to be
  framed in this sense, we call the disk \emph{twisted}.  (Technically
  a twisted Whitney disk may be framed.)
\end{definition}

\begin{definition}[Framed Whitney tower]
  \label{definition:framed-whitney-tower}
  A \emph{framed Whitney tower} in a 4-manifold $X$ is a 2-complex
  defined inductively as follows.  A union of properly immersed
  surfaces in $X$ which are transverse to each other is a \emph{framed
    Whitney tower}.
  Suppose $T$ is a framed Whitney tower and $D$ is an immersed framed
  Whitney disk in the interior of $X$ pairing two intersections of
  opposite signs between two sheets in $T$.  We allow the interior of
  $D$ to transversely intersect the interior of surfaces and disks of
  $T$, but require $D$ to be disjoint from the boundary of any surface
  or disk in~$T$.  Then $T$ with $D$ attached is a \emph{framed
    Whitney tower}.
\end{definition}

\begin{definition}[Order]
  \label{definition:order}
  The initial surfaces of a Whitney tower, namely those with boundary
  in $\partial X$, are called the \emph{order 0 surfaces}.
  Inductively, an intersection of an order $k$ sheet and an order
  $\ell$ sheet is called an \emph{order $k+\ell$ intersection}.  A
  Whitney disk pairing two order $n$ intersections is called an
  \emph{order $n+1$ disk}.  A Whitney tower $T$ is of \emph{order $n$}
  if all intersections of order $<n$ are paired up by Whitney disks
  in~$T$.  (Intersections of order $\ge n$ are allowed to be
  unpaired.)
\end{definition}

\begin{definition}[Twisted Whitney tower]
  \label{definition:twisted-whitney-tower}
  A \emph{twisted Whitney tower of order $n$} is defined exactly in
  the same way as a framed Whitney tower of order $n$, except that we
  allow Whitney disks of order $\ge \tfrac{n}{2}$ to be twisted.
  Disks of order $<\tfrac{n}{2}$ are still required to be framed.
\end{definition}

A twisted Whitney tower of order $n$ can be modified in such a way
that all Whitney disks of order $>\tfrac{n}{2}$ are framed, by a
boundary twist argument (see
\cite[Section~4.1]{Conant-Schneiderman-Teichner:2012-2}).  Using
this, we always assume that a twisted Whitney tower of order $2k-1$
is indeed framed, and assume that a twisted Whitney disk of a
twisted Whitney tower of order $2k$ has order~$k$.

Following the convention of
Freedman-Quinn~\cite{Freedman-Quinn:1990-1} used for gropes, we call a
(framed or twisted) Whitney tower \emph{union-of-disks-like}
(respectively \emph{union-of-annuli-like}) if each order zero surface
is a disk (respectively an annulus).  As mentioned at the beginning of
this section, we assume that every Whitney tower is
union-of-disks-like unless stated otherwise.

We remark that a Whitney tower can always be modified, using finger
moves, in such a way that for each Whitney disk $D$ (except the base
disks or annuli) one of the following holds: (i)~$D$ is a twisted disk
with $\omega(D)=\pm1$, (ii)~$D$ is a framed disk with exactly one
intersection point, or (iii)~$D$ is a framed disk with exactly two
intersection points and they are paired by some other Whitney
disk~\cite[Lemma~2.12]{Conant-Schneiderman-Teichner:2012-2}.  Such a
tower is called \emph{split}.  We always assume that a Whitney tower
is split, unless stated otherwise.

In this paper, links are always oriented and ordered.

\begin{definition}[Boundary of Whitney towers]
  \label{definition:link-bounding-whitney-tower}
  Suppose $X$ is a 4-manifold and $L$ is a framed link
  in~$\partial X$.  We say that $L$ \emph{bounds an order $n$ framed
    Whitney tower $T$ in $X$} if (i)~the boundary of the order zero
  disks of $T$ is equal to~$L$, and (ii)~the unique framing of the
  order zero disks restricts to the given framing of~$L$.  For $n=0$,
  we say that $L$ \emph{bounds an order $0$ twisted Whitney tower $T$
    in $X$} if (i) holds, without requiring~(ii).  For $n>0$, a framed
  link $L$ \emph{bounds an order $n$ twisted Whitney tower $T$ in $X$}
  if (i) and (ii) hold.
\end{definition}

When a framed link $L$ bounds an order 0 Whitney tower $T$, the
\emph{twisting number} $\omega(D)\in\Z=\operatorname{SO}(2)$ of an
order 0 disk of $T$ is defined to be the disk framing with respect to
the given framing of~$L$.

In Definition~\ref{definition:link-bounding-whitney-tower}, we require
the framing condition (ii) even for the twisted case when $n>0$,
because we always regard order $<\frac n2$ surfaces as framed, as we
did in Definition~\ref{definition:twisted-whitney-tower}.  The same
happens in the following definition.

\begin{definition}[Whitney tower concordance]
  \label{definition:whitney-tower-concordance}
  Suppose $X$ is a 4-manifold with
  $\partial X = \partial_+X \sqcup -\partial_-X$.  Two framed links
  $L\subset \partial_-X$ and $L'\subset \partial_+X$ with $m$
  components are \emph{order $n$ framed Whitney tower concordant in
    $X$} if there is a union-of-annuli-like framed Whitney tower $T$
  of order $n$ in $X$ such that (i)~$T$ has $m$ order zero annuli and
  the $i$th order zero annulus is cobounded by the $i$th component of
  $L'$ and that of $-L$, and (ii)~the framings of $L$ and $L'$ extend
  to the same framing of the order zero annuli.  For $n>0$, $L$ and
  $L'$ are \emph{order $n$ twisted Whitney tower concordant in $X$} if
  there is a twisted Whitney tower $T$ of order $n$ satisfying (i)
  and~(ii).  For $n=0$, $L$ and $L'$ are \emph{order $0$ twisted
    Whitney tower concordant} if there is a twisted Whitney tower $T$
  satisfying~(i).
\end{definition}

\begin{remark}[Framing of the boundary of rational Whitney towers]
  \label{remark:links-bounding-towers-in-QHD4}
  In this paper we will mainly consider the case of a Whitney tower in
  a rational homology 4-ball bounded by $S^3$ (or standard $D^4$ as a
  special case) and a Whitney tower concordance in a rational homology
  $S^3\times I$ bounded by $S^3\times 1\sqcup -S^3\times 0$.  Recall
  that the linking number of two knots in $S^3$ is equal to the
  algebraic intersection number of bounding immersed disks in a
  rational homology 4-ball bounded by~$S^3$.  The following basic
  observations are direct consequences of this fact and the above
  definitions.

  \begin{enumerate}
  \item A framed link $L\subset S^3$ bounds an order 0 framed Whitney
    tower in a rational homology 4-ball if and only if each component
    of $L$ is evenly framed, since an immersed disk in a rational
    homology 4-ball bounded by a knot $K\subset S^3$ induces an even
    framing on~$K$.  On the other hand, any framed link $L\subset S^3$
    bounds an order 0 twisted Whitney tower in~$D^4$.
  \item If a framed link $L\subset S^3$ bounds a framed/twisted
    Whitney tower of order $n\ge 1$ in a rational homology 4-ball,
    then the link is automatically zero framed and any two components
    have vanishing linking number; it follows from the fact that all
    the intersections of order zero disks are paired up by order 1
    disks.
  \item If two framed links in $S^3$ are order $n\ge 1$ Whitney tower
    concordant in a rational homology $S^3\times I$, then their
    framings are equal.
  \end{enumerate}
  We will often say, e.g., ``$L\subset S^3$ bounds an order $n\ge 1$
  twisted/framed Whitney tower in a rational homology 4-ball'' even
  when no framing on $L$ is given.  Using~(2), this is understood as
  that $L$ with the zero framing does.
\end{remark}




\subsection{Trees from intersection and twisting data}
\label{subsection:trees-intersection-twisting}

In this subsection we review a certain type of trees used
in~\cite{Conant-Schneiderman-Teichner:2012-2,
  Conant-Schneiderman-Teichner:2012-3}.  Fix an integer $m\ge 1$.  In
this paper trees will always be uni-trivalent and oriented, that is,
each vertex is either univalent or trivalent, and each trivalent
vertex is endowed with a cyclic ordering of adjacent edges.  As a
convention, in a local planar diagram of a vertex and its adjancent
edges, the edges are always ordered counterclockwise.  A tree has
\emph{order $n$} if it has $n$ trivalent vertices.  A tree is
\emph{decorated} if each univalent vertex has a label in
$\{1,\ldots,m\}$.  For a rooted tree, namely when the tree has a
distinguished univalent vertex, it is \emph{decorated} if each
non-root vertex has a label in $\{1,\ldots,m\}$.  In this paper trees
are always decorated.  For two rooted trees $t$ and $t'$, the
\emph{inner product} $\langle t,t''\rangle$ is defined by joining the
roots of $t$ and~$t'$.  The order of $\langle t,t''\rangle$ is the sum
of the orders of $t$ and~$t''$.  Sometimes (but not always) we will
label the root of a rooted tree by the symbol~$\tw$; such a tree is
called a \emph{$\tw$-tree}.

Suppose $T$ is a twisted Whitney tower of order~$n$.  Fix an
orientation of each disk in~$T$, and fix an order of the order 0
disks.  First, we associate to each disk $D$ in $T$ a rooted
tree~$t_D$ as follows.  For the $i$th order 0 disk $D$, define
$t_D = \smallItree{i}$, a rooted tree of order 0 with the non-root
vertex labeled by~$i$.  For a twisted/framed Whitney disk $D$ of order
$>0$, if $D$ pairs two intersections between two disks $D'$ and~$D''$,
define $t_D=\smallrootedYtree{t_{D'}}{t_{D''}}$, that is, the rooted
tree of order 1 with $t_{D'}$ and $t_{D''}$ attached to the leaves.
Here, $D'$ and $D''$ are chosen in such a way that if one travels
along the Whitney circle $\partial D$ near the involved negative
intersection $p$ of $D'$ and $D''$, starting from $D'$, passing
through $p$ and then entering into $D''$, then it agrees with the
orientation of $\partial D$ induced by the given orientation of~$D$.

For each unpaired intersection $p$ in the tower $T$, if
$p\in D \cap D'$, then define $t_p=\langle t_D,t_{D'}\rangle$.  (When
we want to remember where the roots of $t_D$ and $t_{D'}$ were, we
draw the edge of $t_p$ containing the original roots as %
$\vcenter{\hbox{
    \begin{tikzpicture}[
      x=1pt,y=1pt,
      over/.style={draw=white,double=black,double distance=.8pt,line width=1.2pt}
      ]
      \draw[thick] (0,0) -- +(120:5) (0,0) -- + (240:5) (0,0) --
      (10,0);
      \draw[over] (10,0) ellipse(1.5 and 2.5);
      \draw[over] (10,0)--(20,0) (20,0) -- +(60:5) (20,0) -- +(-60:5);
    \end{tikzpicture}}}$\kern.2em%
; the small enclosing circle denotes the location of~$p$.)  Note that
$t_D$ and $D$ have the same order, and therefore so do $t_p$ and~$p$.
For each intersection $p$, denote the sign of $p$ by
$\epsilon(p)=\pm1$.

For each twisted Whitney disk $D$, let $t^\tw_D$ be the tree $t_D$
with the root labeled by~$\tw$, as a $\tw$-tree.  Recall that
$\omega(D)$ is the twisting number of $D$ (see
Definition~\ref{definition:twisted-framed-whitney-disk}).
  
\begin{definition}
  \label{definition:tree-invariant-of-twisted-tower}
  For a twisted Whitney tower $T$ of order $n$, define a formal sum
  $t^\tw_n(T)$ of trees by
  \[
    t^\tw_n(T) = \sum_{p} \epsilon(p)\cdot t_p + \sum_D \omega(D)
    \cdot t^\tw_D
  \]
  where $p$ varies over the order $n$ intersections and $D$ varies
  over the twisted Whitney disks of order~$n/2$.  The second sum is
  regarded as vacuous if $n$ is odd.  Note that unpaired intersections
  of order $>n$ are ignored in~$t^\tw_n(T)$.
\end{definition}

\section{Milnor invariants and rational Whitney towers}
\label{section:milnor-invariant-rational-tower}

In this section we prove the following relationship of Milnor
invariants of links and Whitney towers in rational homology 4-balls.

\begin{theorem}
  \label{theorem:whitney-tower-milnor-invariant}
  Suppose $L$ is a framed link in $S^3$ bounding a twisted Whitney
  tower $T$ of order $n\ge 0$ in a rational homolgy 4-ball bounded
  by~$S^3$.  Then the following hold.
  \begin{enumerate}
  \item $L$ has vanishing Milnor invariants of order $< n$
    \textup{(}or equivalently length $< n+2$\textup{)}.
  \item $T$ determines the order $n$ Milnor invariant of~$L$.  In
    fact, $\mu_n(L) = \eta_n(t_n^\tw(T))$.
  \end{enumerate}
\end{theorem}

In Theorem~\ref{theorem:whitney-tower-milnor-invariant}~(2), $\mu_n$
denotes the total Milnor invariant of order $n$, and $\eta_n$ denotes
the summation map which was formulated in
\cite{Levine:2001-1,Levine:2001-2} and used extensively
in~\cite{Conant-Schneiderman-Teichner:2012-2,
  Conant-Schneiderman-Teichner:2012-3,
  Conant-Schneiderman-Teichner:2012-1}.  We will review their
definitions in Section~\ref{subsection:review-milnor-invariant}.

Theorem~\ref{theorem:whitney-tower-milnor-invariant}
generalizes~\cite[Theorem~6]{Conant-Schneiderman-Teichner:2012-2},
which states the same conclusion under a weaker hypothesis that $T$ is
\emph{in~$D^4$}.  We remark that the proof of
\cite[Theorem~6]{Conant-Schneiderman-Teichner:2012-2} first converts
the given Whitney tower to a capped grope and then works with the
resulting grope, particularly using the grope duality of Krushkal and
Teichner~\cite{Krushkal:1997-1}.  In our proof of
Theorem~\ref{theorem:whitney-tower-milnor-invariant}, we present a
Whitney tower argument inspired by the grope argument
in~\cite{Conant-Schneiderman-Teichner:2012-2}.  We wish this
alternative approach, which works for Whitney towers in $D^4$ as well,
to be a useful addition to the literature.

As a part of our proof of
Theorem~\ref{theorem:whitney-tower-milnor-invariant}, we show a Milnor
type theorem for Whitney towers in a rational homology 4-ball.  See
Theorem~\ref{theorem:milnor-theorem-rational-tower} in
Section~\ref{section:milnor-type-theorem-for-tower}.  Its analog for
capped gropes in $D^4$ appeared earlier
in~\cite[Lemma~33]{Conant-Schneiderman-Teichner:2012-3}.

\subsection{A quick review on the Milnor invariant and summation map}
\label{subsection:review-milnor-invariant}

We begin by recalling the definition of the Milnor invariant and
summation map, and setting up notations.

In the original work of Milnor~\cite{Milnor:1957-1}, the invariant is
defined modulo certain indeterminancy to handle arbitrary links, but
we will consider only the special case that it is well defined without
indeterminancy.

Denote the lower central series of a group $\pi$ by $\{\pi_{k}\}$,
which is defined inductively by $\pi_1=\pi$,
$\pi_{k+1} = [\pi,\pi_{k}]$.  In this paper, we use the convention
$[a,b]=aba^{-1}b^{-1}$.

Suppose $L$ is an $m$-component link in~$S^3$ with
$\pi=\pi_1(S^3\sm L)$.  Let $\mu_i\in \pi$ and $\lambda_i\in \pi$ be
the class of a meridian and a zero linking longitude of the $i$th
component respectively.  Let $F$ be a free group generated by
$x_1,\ldots,x_m$.  Let $F\to \pi$ be the meridian map defined by
$x_i \mapsto \mu_i$.  Suppose $n\ge 0$, and suppose $\lambda_i$ is
contained in~$\pi_{n+1}$.  (It is always the case for $n=0$.)  Then,
by Milnor~\cite[Theorem~4]{Milnor:1957-1}, $F\to \pi$ induces an
isomorphism $\pi_{n+1}/\pi_{n+2} \cong F_{n+1}/F_{n+2}$%
.  Let $w_i\in F_{n+1}/F_{n+2}$ be the image of $\lambda_i$ under the
isomorphism.  The \emph{Milnor invariant of length $n+2$} can be
defined to be the $m$-tuple $(w_1,\ldots,w_m)\in (F_{n+1}/F_{n+2})^m$.
If the Milnor invariant of length $n+2$ vanishes, then the longitudes
$\lambda_i$ lie in $\pi_{n+2}$, so that the Milnor invariants of
length $n+3$ can be defined.

Summarizing the above, the Milnor invariant of length $n+2$ is defined
(without indeterminancy) when the Milnor invariants of length
$\le n+1$ vanish, and it is the case if and only if every longitude
lies in the lower central subgroup~$\pi_{n+1}$.

From the longitude elements $w_i\in F_{n+1}/F_{n+2}$, Milnor extracted
numerical invariants denoted by $\ol\mu_L(i_1,\ldots,i_{n+1},i)$ for
$1\le i_j\le m$, via the Magnus expansion.  (This is why it is called
of \emph{length}~$n+2$.)  For our purpose, following
\cite{Conant-Schneiderman-Teichner:2012-3,
  Conant-Schneiderman-Teichner:2012-2}, it is convenient to use the
the free Lie algebra $\sL$ generated by $m$ variables
$X_1,\ldots,X_m$.  We have $\sL=\bigoplus_n \sL_n$ where $\sL_n$ is
the degree $n$ part; $\sL_n$ is equal to the quotient of the free
abelian group generated by $n$-fold brackets in $X_1,\ldots,X_m$
modulo the Jacobi relation and alternativity relation $[X,X]=0$.  In
particular, $\sL_1$ is the free abelian group generated by
$X_1,\ldots,X_m$.  It is known that the association $x_i \mapsto X_i$
gives rise to an isomorphism $F_n/F_{n+1} \to \sL_n$ which takes
commutator brackets to Lie brackets.  For instance see
\cite[Section~5.7]{Magnus-Karrass-Solitar:1966-1}.  Let $u_i$ be the
image of $w_i$ under $F_{n+1}/F_{n+2} \to \sL_{n+1}$.  The \emph{total
  Milnor invariant of order $n$} is defined by
\[
  \mu_{n}(L) = \sum_{i=1}^m X_i \otimes u_i \in \sL_1 \otimes
  \sL_{n+1}.
\]
Note that order $n$ corresponds to length~$n+2$.

Let $\sD_{n}$ be the kernel of the bracket map
$\sL_1 \otimes \sL_{n+1} \to \sL_{n+2}$ defined by
$X\otimes Y \to [X,Y]$.  Milnor's cyclic
symmetry~\cite[Theorem~5]{Milnor:1957-1} implies that
$\mu_{n}(L)\in \sD_{n}$ for any link~$L$.  Moreover, as a function of
the set of links $L$ with $\mu_{k}(L)=0$ for $k\le n-1$, $\mu_{n}$ is
surjective onto~$\sD_n$.  It is a consequence
of~\cite[Theorem~6]{Conant-Schneiderman-Teichner:2012-3} and
\cite[Theorem~1]{Levine:2001-2}.

\begin{remark}[Rank of $\sD_n$]
  \label{remark:rank-of-D_n}
  The range $\sD_n$ of $\mu_n$ is a free abelian group of known rank.
  Due to Witt (e.g., see
  \cite[Section~5.6]{Magnus-Karrass-Solitar:1966-1}), $\sL_k$ is a
  free abelian group of rank~$\cR(m,k)$, where
  $\cR(m,k)=\frac 1k \sum_{d\mid k} \phi(d)\cdot m^{k/d}$ with
  $\phi(d)$ the M\"obius function, as already given in
  Theorem~\ref{theorem:classification-intro} in the introduction.  It
  follows that $\sD_n$ is a free abelian group of rank
  \[
    \cM(m,n) := m \cR(m,n+1)-\cR(m,n+2).
  \]
  It was first shown by Orr~\cite{Orr:1989-1} that $\cM(m,n)$ is the
  number of linearly independent Milnor invariants of length $n+2$ on
  links $L$ with vanishing Milnor invariants of length $\le n+1$.
\end{remark}

\begin{remark}[Independence from meridian/longitude choices]
  \label{remark:mu-independence-of-meridian-longitude-choice}
  For any $L$ with $\mu_q(L)=0$ for $q\le n-1$, $\mu_n(L)$ is
  well-defined, independent of the choice of meridians $\mu_i$ (i.e.,
  the meridian map $F\to \pi$).  It is essentially because two
  meridians are conjugate: if a meridian map is given by
  $x_i \mapsto \mu_i$, then another meridian map is of the form
  $x_i \mapsto g_i \mu_i g_i^{-1}$, and it is straightforward to
  verify that they induce the same homomorphism
  $F_{n+1}/F_{n+2} \to \pi_{n+1}/\pi_{n+2}$, by using standard
  commutator calculus.  Also, $\mu_n(L)$ is independent of the choice
  of longitudes $\lambda_i$, since any conjugate of
  $\lambda_i\in \pi_{n+1}$ is equal to $\lambda_i$ itself
  modulo~$\pi_{n+2}$.
\end{remark}

\begin{remark}[Milnor invariant for framed links]
  \label{remark:milnor-invariant-framed-link}
  For a \emph{framed} link $L \subset S^3$, we define the Milnor
  invariant using pushoffs of components taken along the given
  framing, instead of zero linking longitudes.  Then, for $n=0$,
  $\mu_0(L)$ is equivalent to the pairwise linking numbers and framing
  of each component.  (In the unframed case $\mu_0(L)$ is equivalent
  to the pairwise linking numbers.)  In particular, $L$ with
  $\mu_0(L)=0$ is automatically zero framed.  Since we always assume
  that $\mu_q(L)=0$ for $q<n$ whenever we consider $\mu_n(L)$, it
  follows that there is no difference between framed and unframed
  cases for $\mu_n(L)$ with $n\ge 1$. 
  Our definition for framed links will be useful in describing order 0
  Whitney tower classifications of links in terms of~$\mu_0$.
\end{remark}

We finish this subsection with the definition of the summation
$\eta_n$ which appeared in the statement of
Theorem~\ref{theorem:whitney-tower-milnor-invariant}.  Recall that a
rooted tree $t$ of order $n$ decorated by $\{1,\ldots,m\}$ determines
a formal $n$-fold bracket in the variables $X_1,\ldots,X_m$, which we
denote by $B(t)$, in the standard manner:
$B(\vcenter{\hbox{\kern.2em\tikz[x=1pt,y=1pt]{\draw[thick,fill] (0,0)
      circle(.6pt)--(0:10) node[right]{$i$}}}}) = X_i$,
$B\big(\smallrootedYtree{t'}{t''}\big) = [B(t'), B(t'')]$.
We will often denote by $B(t)$ the element in $\sL_{n+1}$ represented
by the bracket~$B(t)$.  For a univalent vertex $v$ of a tree $t$, let
$t_v$ be the rooted tree obtained by deleting the decoration of $v$
and taking $v$ as the root.
  
\begin{definition}[Summation $\eta_n$]
  \label{definition:summation-map}
  For a tree $t$ of order $n$, define
  $\eta_n(t)=\sum_v X_{\ell(v)} \otimes B(t_v) \in
  \sL_1\otimes\sL_{n+1}$ where $v$ varies over all the univalent
  vertices of $t$ and $\ell(v)$ is the decoration of~$v$.  When $n$ is
  even, define $\eta_n$ for a $\tw$-tree $t^\tw$ of order~$\frac n2$
  by
  $\eta_n(t^\tw) = \frac12 \eta_n(\langle t^\tw,t^\tw\rangle) \in
  \sL_1\otimes\sL_{n+1}$.  It is straightforward to verify that
  $\eta_n(t^\tw)$ has integer coefficients.  For a formal sum of
  decorated order $n$ trees, and in addition order $\frac n2$
  $\tw$-trees $t^\tw$ when $n$ is even, define $\eta_n$ by extending
  the above linearly.
\end{definition}

\subsection{Computing meridians and Whitney circles in a Whitney tower}
\label{subsection:meridian-longitude-in-whitney-tower}

In this subsection we discuss how to compute Whitney circles and
meridians of Whitney disks in the fundamental group of a Whitney
tower complement using commutator calculus.

In what follows, the order 0 disks of a Whitney tower $T$ in a
4-manifold $X$ are always ordered.  For a formal $r$-fold bracket $B$
in $X_1,\ldots,X_m$ with $r\ge k+1$, we also denote by the same symbol
$B$ the element in $\pi_1(X\sm T)_{k+1}/\pi_1(X\sm T)_{k+2}$ obtained
by substituting a meridian of the $i$th order 0 disk for each
occurance of $X_i$ in the formal bracket~$B$.  This element is
well-defined modulo $\pi_1(X\sm T)_{k+2}$, independent of the choice
of a meridian, as in
Remark~\ref{remark:mu-independence-of-meridian-longitude-choice}.  It
is trivial in $\pi_1(X\sm T)_{k+1}/\pi_1(X\sm T)_{k+2}$ if $r>k+1$.

The following lemma says that the meridian of a Whitney disk $D$ is
essentially the commutator associated to the tree~$t_D$.

\begin{lemma}[Commutator expression of a meridian]
  \label{lemma:meridian-in-whitney-tower}
  Suppose $T$ is a twisted Whitney tower in a 4-manifold~$X$, and $D$
  is an order $k$ disk in~$T$.  Then a meridian $\mu_D$ of $D$ lies in
  $\pi_1(X\sm T)_{k+1}$ and $\mu_D=B(t_D)$ in
  $\pi_1(X\sm T)_{k+1} / \pi_1(X\sm T)_{k+2}$.
\end{lemma}

\begin{proof}
  We use an induction on~$k$.  For $k=0$, the conclusion is
  straightforward.  Suppose $D$ is an order $k$ disk with $k\ge 1$ and
  the conclusion holds for order $<k$.  Since $t_D$ has order $k$,
  $B(t_D)$ is a $(k+1)$-fold bracket.  So it suffices to show that
  the meridian $\mu_D$ is of the form~$B(t_D)$.  The Whitney disk $D$
  pairs intersections of two disks $D'$ and $D''$ of order $r$ and $s$
  with $r+s+1=k$ by definition.  The meridians $\mu_{D'}$ of $D'$ and
  $\mu_{D''}$ of $D''$ are standard basis curves of a Clifford torus
  around the involved negative intersection.  Since the Clifford torus
  meets $D$ at a single transverse intersection, $\mu_D$ is equal to a
  commutator of $\mu_{D'}$ and~$\mu_{D''}$.  In fact, choosing $D'$
  and $D''$ in such a way that
  $t_D= \smallrootedYtree{t_{D'}}{t_{D''}}$ holds (see the orientation
  convention in Section~\ref{subsection:trees-intersection-twisting}),
  we have $\mu_D=[\mu_{D'},\mu_{D''}]$.  Since $\mu_{D'}=B(t_{D'})$
  and $\mu_{D''}=B(t_{D''})$ by the induction hypothesis, we have
  $\mu_D=[B(t_{D'}),B(t_{D''})] = B(t_D)$ as desired.
\end{proof}

To compute the Whitney circles, we will use the following notations.
Recall that we assume that a Whitney tower is split, and a twisted
Whitney disk $D$ in a Whitney tower of order $n$ has order~$n/2$ and
$\omega(D)=\pm1$.

\begin{definition}[Complementary tree $t^c_D$ of a Whitney disk~$D$]
  \label{definition:complementary-tree}
  Suppose $D$ is a Whitney disk in an order $n$ twisted Whitney
  tower~$T$.  If $D$ contains two paired intersections, proceed to the
  next stage Whitney disk that pairs the intersections.  Repeating
  this, one eventually reaches either a framed Whitney disk with an
  unpaired intersection $p$ of order $\ge n$, or a twisted Whitney
  disk~$D'$ of order $\frac n2$.  Let $t^u_D := t_p$ in the former
  case and let $t^u_D := \langle t^\tw_{D'},t^\tw_{D'}\rangle$ in the
  latter case.  Our $t^u_D$ contains $t_D$ as a subtree; the root of
  $t_D$ is the midpoint of an edge of~$t^u_D$.  When
  $t^u_D = \langle t^\tw_{D'},t^\tw_{D'}\rangle$, we just fix one of
  the two copies of $t_D$ in~$t^u_D$.  Define the \emph{complementary
    tree} $t^c_D$ of~$D$ to be $t^u_D$ with $t_D$ removed, with the
  root of $t_D$ as the root of~$t^c_D$.  Define the
  \emph{complementary sign} $\epsilon^c_D$ to be the sign
  $\epsilon(p)$ of $p$ if $t^u_D=t_p$, and to be the twisting number
  $\omega(D')$ if $t^u_D = \langle t^\tw_{D'},t^\tw_{D'}\rangle$.
\end{definition}

If $D$ is an order $k$ disk in a Whitney tower of order $n$, then the
complementary tree $t^c_D$ of has order $\ge n-k$, since $t^u_D$ has
order~$\ge n$.

\begin{lemma}[Commutator expression of a Whitney circle]
  \label{lemma:disk-boundary-in-lower-central-series}    
  Suppose $T$ is an order $n$ twisted Whitney tower in a
  4-manifold~$X$ and $D$ is an order $k$ Whitney disk in~$T$ with
  $k>0$.  Let $\gamma_D$ be a pushoff of the Whitney circle
  $\partial D$, taken along the Whitney framing.  Then $\gamma_D$ lies
  in~$\pi_1(X\sm T)_{n-k+1}$, and $\gamma_D = B(t^c_D)^{\epsilon^c_D}$
  in $\pi_1(X\sm T)_{n-k+1}/\pi_1(X\sm T)_{n-k+2}$.
\end{lemma}

It follows from
Lemma~\ref{lemma:disk-boundary-in-lower-central-series} that $\gamma_D$
is trivial in $\pi_1(X\sm T)_{n-k+1}/\pi_1(X\sm T)_{n-k+2}$ if the
complementary tree $t^c_D$ has order~$>n-k$, or equivalently $t^u_D$
has order~$>n$.
  
\begin{proof}[Proof of Lemma~\ref{lemma:disk-boundary-in-lower-central-series}]
  
  Let $G=\pi_1(X\sm T)$.  As a special case, suppose $D$ is a framed
  disk which has an order $\ge n$ unpaired intersection $p$ with
  another disk~$D'$.  Then
  $\gamma_D=(\mu_{D'})^{\epsilon_p}= B(t_{D'})^{\epsilon^c_D}$ by
  Lemma~\ref{lemma:meridian-in-whitney-tower}.
  Since $t^u_D = t_p = \langle t_D,t_{D'}\rangle$,
  $t^c_D = t_{D'}$.  It follows that
  $\gamma_D = B(t^c_D)^{\epsilon^c_D}$ as claimed.  As another special
  case, suppose $D$ is a twisted Whitney disk.  Then since $\gamma_D$
  is taken along the Whitney framing,
  $\gamma_D\cdot(\mu_{D})^{-\omega(D)}$ bounds a parallel of~$D$.
  (Note that the exponent $-\omega(D)$ represents the disk framing
  with respect to the Whitney framing, since
  $\omega(D)\in \Z=\operatorname{SO}(2)$ is defined to be the Whitney
  framing with respect to the disk framing.) Therefore
  $\gamma_D = (\mu_D)^{\omega(D)} = B(t_D)^{\epsilon^c_D}$ by
  Lemma~\ref{lemma:meridian-in-whitney-tower}.  Since
  $t^u_D = \langle t^\tw_D,t^\tw_D\rangle$, $t^c_D$ is $t_D$ itself.
  It follows that $\gamma_D = B(t^c_D)^{\epsilon^c_D}$.

  Now we proceed inductively, from higher to lower stage Whitney
  disks, using the above cases as the initial step.  Suppose $D$ is a
  Whitney disk of order~$k$.  If $D$ is not one of the above two
  special cases, then $D$ is a framed disk with two intersections with
  another disk $D'$ and the intersections are paired by a next stage
  Whitney disk~$D''$.  The induction hypothesis is that
  $\gamma_{D''} = B(t^c_{D''})^{\epsilon^c_{D''}}$.

  We have either $t_{D''}=\smallrootedYtree{t_{D}}{t_{D'}}$
  or~$\smallrootedYtree{t_{D'}}{t_{D}}$.  We will present details only
  for the former case, since the argument applies to the latter case
  in the essentially same way.
  Figure~\ref{figure:disk-boundary-computation} shows the disks $D$,
  $D'$ and $D''$ when $t_{D''}=\smallrootedYtree{t_{D}}{t_{D'}}$.  The
  circular arrows near $\partial D$ and $\partial D''$ specify the
  orientations of $D$ and~$D''$.  The disk $D'$ is oriented in such a
  way that~$\mu_{D'}$ is a positively oriented meridian.  In
  Figure~\ref{figure:disk-boundary-computation}, the negatively
  oriented meridian of $D'$ which is near the $-$ intersection is
  equal to~$\gamma_{D''}^{}\mu_{D'}^{-1}\gamma_{D''}^{-1}$.  (Here one
  may use a basepoint near the $+$ intersection.)  Therefore the
  pushoff $\gamma_D$ of $\partial D$ is the product of $\mu_{D'}$ and
  $\gamma_{D''}^{\mathstrut}\mu_{D'}^{-1}\gamma_{D''}^{-1}$.  By
  Lemma~\ref{lemma:meridian-in-whitney-tower} and the induction
  hypothesis,
  $\gamma_D = [\mu_{D'},\gamma_{D''}] =
  [B(t_{D'}),B(t^c_{D''})^{\epsilon^c_{D''}}]$.  Using
  $\epsilon^c_{D''} = \epsilon^c_D$, and using
  $[a,b^{-1}] = b^{-1}[a,b]^{-1}b$ when $\epsilon^c_D=-1$, we obtain
  $\gamma_D = [B(t_{D'}),B(t^c_{D''})]^{\epsilon^c_D}$
  in~$G_{n-k+1}/G_{n-k+2}$.  Since
  $t_{D''}=\smallrootedYtree{t_{D}}{t_{D'}}$, the complementary tree
  of $D$ is given by $t^c_D = \smallrootedYtree{t_{D'}}{t^c_{D''}}$.
  It follows that
  $B(t^c_D)^{\epsilon^c_D} = [B(t_{D'}), B(t^c_{D''})]^{\epsilon^c_D}
  = \gamma_D$ as promised.
\end{proof}

\begin{figure}[H]
  \begin{tikzpicture}
    [x=5bp,y=5bp,scale=.75,line width=1pt,double distance=1pt,join=round,
    over/.style={draw=white,double=black,double distance=1pt,line width=1.5pt},
    thinover/.style={draw=white,double=black,double distance=.4pt,line width=1.1pt},
    ->-/.style={draw=black,decoration={
        markings,mark=at position #1 with {\arrow{stealth}}},postaction={decorate}},
    >=stealth,
    hidden/.style={text=white,opacity=0}
    ]
    \clip (0,-4) rectangle (50,27);
    \small
    \def\c[#1]#2{coordinate(#2) node[hidden,#1]{\footnotesize$#2$}}
    
    \draw[cap=round] (10,16) \c[above]{A} -- ++(40,0) \c[above]{A1}
    -- ++(-10,-16) \c[below]{A2};
    \draw[thin,rounded corners=4] ($(A)+(.6,-1)$) \c[below]{a} ++(3,0)
    -- ($(A1)+(-1.8,-1)$) \c[left]{a1} -- ($(A2)+(-.8,1)$) \c[above]{a2} -- ++(-4,0);
    \draw[over] (15,-3) \c[left]{B} -- (15,8) \c[left]{P} ++(20,0) \c[right]{Q} --
    ++(0,-11) \c[right]{B1};
    \draw[->,draw=black] (B1) -- +(0,-1);
    \draw[over] (15,5) \c[left]{C} ellipse(2 and 1) ++(20,0) \c[left]{D}
    ellipse(2 and 1);
    \draw[->-=.75] (15,5) ellipse(2 and 1);
    \draw[->-=.75] (35,5) ellipse(2 and 1);

    \draw[over] (C) -- (P) -- ++(0,2) ..controls+(0,22)and+(0,22)..
    ($(Q)+(0,2)$) -- (Q) -- (D);
    \draw (P) circle(1pt) -- (Q) circle(1pt);

    \draw[over] (A) ++(-10,-16) \c[below]{A3} -- ++(39,0);
    \draw[cap=round] (A) -- ++(-10,-16) -- ++(40,0);
    \draw[thinover,rounded corners=4] (a2) ++(-3,0) -- ($(A3)+(1.8,1)$) \c[above]{a3}
    -- ($(a3)!.8!(a)$);
    \draw[->,thin] ($(a3)!.5!(a)$) -- ($(a3)!.8!(a)$);

    \draw[thinover,rounded corners=2]
    ($(P)+(1,2)$) ..controls+(0,20.5)and+(0,20.5).. ($(Q)+(-1,2)$) --
    ($(Q)+(-1,1)$) -- ($(P)+(5,1)$);
    \draw[<-,thin,draw=black] ($(P)+(5,1)$) -- +(2,0);
    
    \draw
    (A) node[anchor=south]{$D$} (B) node[anchor=west]{$D'$}
    ($ .5*(C)+.5*(D)+(0,16) $) node{$D''$}
    (C) node[anchor=east,xshift=-6pt,yshift=-2pt]{$\mu_{D'}$}
    (D) node[anchor=east,xshift=-6pt,yshift=-2pt]{$\gamma_{D''}^{\mathstrut}\mu_{D'}^{-1}\gamma_{D''}^{-1}$}
    (P) node[anchor=east]{$+$} (Q) node[anchor=west]{$-$};
  \end{tikzpicture}
  \caption{The disks $D$, $D'$ and~$D''$ and the meridian~$\mu_{D'}$.}
  \label{figure:disk-boundary-computation}
\end{figure}
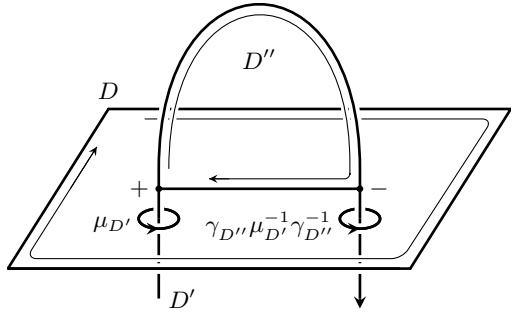

The proof of Lemma~\ref{lemma:disk-boundary-in-lower-central-series}
applies to an order 0 disk $D$ in essentially the same way.  The
statement is as follows.

\begin{lemma}
  \label{lemma:order-zero-disk-boundary}
  Suppose $T$ is a twisted Whitney tower of order~$n\ge 0$ in a
  4-manifold $X$ bounded by a framed link~$L$.  Then the $i$th
  longuitude $\lambda_i$ of $L$ taken along the given framing lies in
  $\pi_1(X\sm T)_{n+1}$.  Furthermore, if the formal sum $t^\tw_n(T)$
  is of the form
  \[
    t^\tw_n(T) = \sum_t a(t) \cdot t + \sum_{t^\tw} b(t^\tw) \cdot t^\tw
  \]
  with $a(t)$, $b(t)\in \Z$, then
  \[
    \lambda_i = \Bigg(\prod_t \prod_{\substack{v\in t \\ \ell(v)=i}}
    B(t_v)^{a(t)} \Bigg) \cdot \Bigg(\prod_{t^\tw}
    \prod_{\substack{u\in t^\tw \\ \ell(u)=i}} B(\langle t^\tw,t^\tw
    \rangle_u)^{b(t^\tw)}\Bigg) \text{ in }
    \frac{\pi_1(X\sm T)_{n+1}}{\pi_1(X\sm T)_{n+2}}.
  \]
  Here $t$ varies over order $n$ trees appearing in $t^\tw_n(T)$, $v$
  varies over the univalent vertices of $t$ with decoration
  $\ell(v)=i$, $t^\tw$ varies over order $\frac n2$ $\tw$-trees
  appearing in $t^\tw_n(T)$, and $u$ varies over univalent vertices of a
  fixed copy of $t^\tw_D$ in $\langle t^\tw_D,t^\tw_D \rangle$ with
  label $\ell(u)=i$.
\end{lemma}

Recall that for a tree $t$ and its univalent vertex $v$, $t_v$ is the
rooted tree obtained by deleting the label of $v$ and taking $v$ as
the root, as we did in Definition~\ref{definition:summation-map}.

\begin{proof}
  Let $D$ be the $i$th order 0 disk of~$T$.  For each order $n$
  unpaired intersection on $D$, choose a disk neighborhood in~$D$.
  For each pair of opposite intersections of $D$ and another disk $D'$
  which are paired by a next stage disk $D''$, choose a disk
  neighborhood of the arc $D\cap D''$ in~$D$.  Denote these disk
  neighborhoods by $U_1,U_2,\ldots\,$; so each $U_j$ contains either
  an order $n$ unpaired intersection or an arc of the form
  $D\cap D''$.  We may assume that the subdisks $U_j$ are mutually
  disjoint.  For each $U_j$, a pushoff $\gamma_j$ of $\partial U_j$ is
  computed by the argument of
  Lemma~\ref{lemma:disk-boundary-in-lower-central-series}.  Here,
  instead of the tree $t^u_D$ used in
  Lemma~\ref{lemma:disk-boundary-in-lower-central-series}, we use
  either an order $n$ tree $t$ appearing in $t^\tw_n(T)$, or
  $\langle t^\tw,t^\tw\rangle$ for some order $\frac n2$ $\tw$-tree
  $t^\tw$ appearing in $t^\tw_n(T)$, which has a univalent vertex $v$
  with label $\ell(v)=i$.  Then $t_v$ or
  $\langle t^\tw,t^\tw\rangle_v$ plays the role of the complementary
  tree.  Therefore, by the argument of
  Lemma~\ref{lemma:disk-boundary-in-lower-central-series}, we obtain
  $\gamma_j = B(t_v)^{\pm1}$ or
  $B(\langle t^\tw,t^\tw\rangle_v)^{\pm1}$, where $\pm$ is the sign of
  the coefficient of $t$ or~$t^\tw$.
  When $n>0$, each univalent vertex $v$ of a tree appearing in
  $t^\tw_n(T)$ with $\ell(v)=i$ is involved in the computation of
  exactly one~$\gamma_j$.  Since a pushoff of the boundary of $D$ is
  equal to $\prod_j \gamma_j$, the promised formula for $\lambda_i$
  follows.  When $n=0$, the situation is indeed simpler but a minor
  change is needed.  All intersections on $D$ are unpaired and of
  order 0, and in addition, there may be trees of the form
  $t^\tw=\smalltwItree{i}$ in $t^\tw(T)$, which is not involved in the
  computation for any~$D_j$ but yields a $\pm$ twisting for $D$.
  Nonetheless, the contribution of such a twisting to $\lambda_i$ is
  $\mu_{D}^{\pm1} = B(\smallItree{i})^{\pm1} = B(\langle
  t^\tw,t^\tw\rangle_v)^{\pm1}$, where $\mu_{D}$ is a meridian of the
  $i$th order zero disk~$D$.  Therefore the claimed formula for
  $\lambda_i$ holds.
\end{proof}

\subsection{A Milnor type theorem for rational Whitney towers}
\label{section:milnor-type-theorem-for-tower}

Define the \emph{rational lower central subgroups} $G^\Q_k$ ($k\ge 1$)
of a group $G$ by $G^\Q_1 := G$ and
\[
  G^\Q_{k+1} = \Ker\Big\{G^\Q_k \to \frac{G^\Q_k}{[G,G^\Q_k]} \to
  \frac{G^\Q_k}{[G,G^\Q_k]} \otimesover{\Z} \Q \Big\}.
\]
It is straightforward to verify that $G_k\subset G^\Q_k$.

\begin{theorem}[Milnor type theorem for rational Whitney towers]
  \label{theorem:milnor-theorem-rational-tower}
  Suppose $T$ is a twisted Whitney tower of order $n$ in a rational
  homology 4-ball $X$ which is bounded by an $m$-component link $L$
  in~$\partial X$.  Let $F\to \pi_1(X\sm T)$ be a homomorphism of the
  free group $F$ generated by $x_1,\ldots,x_m$ which sends $x_i$ to a
  meridian of the $i$th component of~$L$.  Then for each $k\le n$, it
  induces an isomorphism
  \[
    \frac{\pi_1(X\sm T)^\Q_{k+1}}{\pi_1(X\sm T)^\Q_{k+2}} \otimesover{\Z} \Q
    \cong \frac{F_{k+1}}{F_{k+2}} \otimesover{\Z} \Q.
  \]
\end{theorem}

To prove Theorem~\ref{theorem:milnor-theorem-rational-tower}, we will
use the following homology computation.

\begin{lemma}
  \label{lemma:tower-exterior-homology}
  Suppose $T$ is a twisted Whitney tower of order $n$ in a rational
  homology 4-ball~$X$.  Then
  $\tilde H_i(X\sm T;\Q) \cong H^{3-i}(T,\partial T;\Q)$, and the
  following hold.
  \begin{enumerate}
  \item The meridians of the order 0 disks form a basis for
    $H_1(X\sm T;\Q)$.
  \item $H_2(X\sm T;\Q)$ is spanned by classes of tori which have
    standard basis curves $\alpha$ and $\beta$ such that
    $\alpha\in \pi_1(X\sm T)_k$ and $\beta\in \pi_1(X\sm T)_{n-k+2}$
    for some~$k$.
  \end{enumerate}
\end{lemma}

\begin{proof}
  Let $G=\pi_1(X\sm T)$ and let $N$ be a regular neighborhood of $T$
  in~$X$.  Let $\partial_-N := \partial N \cap \partial X$ and
  $\partial_+N := \ol{\partial N \sm \partial_-N}$.  Then
  
  \begin{equation}
    \label{equation:homology-duality-computation}
    \begin{aligned}
      \tilde H_i(X\sm T;\Q)
      &\cong H_{i+1}(X,X\sm T;\Q) && \text{since }\tilde H_*(X;\Q)=0,\\
      &\cong H_{i+1}(N,\partial_+ N;\Q) && \text{by excision,} \\
      &\cong H^{3-i}(N,\partial_- N;\Q) && \text{by duality for }(N,\partial_+N,\partial_-N), \\
      &\cong H^{3-i}(T,\partial T;\Q) &&\text{since }(N,\partial_- N) \simeq (T,\partial T).
    \end{aligned}
  \end{equation}
  Since $H_2(T,\partial T)$ is the free abelian group generated by the
  fundamental classes of the order zero disks rel boundary, the
  meridians of the order zero disks, which are dual to the fundamental
  classes, form a basis of $H_1(X\sm T;\Q)$.  This proves~(1).
  
  The remaining part is devoted to the proof of~(2).  Let $m$ be the
  number of order zero disks.  The pair $(T,\partial T)$ is homotopy
  equivalent to
  $K:=\big(\bigsqcup_{i=1}^m (D^2,S^1)\big) \cup \big(\bigsqcup_j
  e_j\big)$, where each $e_j$ is a 1-cell attached along a map
  $\partial e_j = S^0 \hookrightarrow \bigsqcup_{i=1}^m \inte D^2$.
  Indeed each $e_j$ is associated to either an unpaired intersection
  of $T$ or a Whitney disk of order~$>0$.  For each $e_j$, we will
  describe a torus $C_j$ which is dual to $e_j$ and has standard basis
  curves $\alpha$ and $\beta$ such that $\alpha\in G_{k}$ and
  $\beta\in G_{n-k+2}$ for some~$k$.  Since the dual tori $C_j$ span
  $H_2(X\sm T;\Q)$ by~\eqref{equation:homology-duality-computation},
  the conclusion~(2) follows.

  \begin{case-named}[Case~1]
    Let $e_j$ be a 1-cell of $K$ associated to an unpaired
    intersection $p$ between two disks $D$ and~$D'$.  In $T$, $e_j$
    corresponds to an arc $\gamma$ in $T$ from $D$ to $D'$
    through~$p$.  See Figure~\ref{figure:arc-unpaired-intersection}.
    Let $C_j$ be the Clifford torus around~$p$.  The torus $C_j$ is
    dual to the 1-cell~$e_j$.  A meridian $\mu$ of $D$ and a meridian
    $\mu'$ of $D'$ are standard basis curves of~$C_j$.  Let $r$ and
    $s$ are the orders of $D$ and $D'$ respectively.  Then
    $\mu\in G_{r+1}$ and $\mu'\in G_{s+1}$ by
    Lemma~\ref{lemma:meridian-in-whitney-tower}.  Since the
    intersection $p$ is left unpaired, $r+s\ge n$.  This shows that
    $C_j$ satisfies the promised property.
  \end{case-named}

  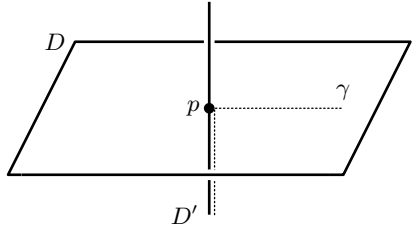
\begin{figure}[H]
    \begin{tikzpicture}
      [x=5bp,y=5bp,scale=1,line width=1pt,double distance=1pt,join=round,
      over/.style={draw=white,double=black,double distance=1pt,line width=1.5pt},
      dotted/.style={line width=.5pt,densely dotted,cap=round},
      hidden/.style={text=white,opacity=0}
      ]
      \small
      \def\c[#1]#2{coordinate(#2) node[hidden,#1]{\footnotesize$#2$}}
      \draw (1,0) \c[below]{A} -- (0,0) -- ++(5,10) node[anchor=east]{$D$}
      -- ++(25,0) -- ++(-5,-10) -- ++(-1,0) \c[below]{B};
      \draw[over] (15,-3) node[anchor=east]{$D'$} -- ++(0,16);
      \draw[fill] (15,5) circle(1.5pt) node[anchor=east]{$p$};
      \draw[dotted] ($ (15,-3)+(2pt,0) $) |- (25,5) node[anchor=south]{$\gamma$};
      \draw[over] (A)--(B);
    \end{tikzpicture}
    \caption{The arc $\gamma$ passing through an unpaired
      intersection~$p$.}
    \label{figure:arc-unpaired-intersection}
  \end{figure}

  \begin{case-named}[Case~2]
    Let $e_j$ be a 1-cell of $K$ associated to a Whitney disk $D''$
    between two disks $D$ and~$D'$.  In $T$, $e_j$ corresponds to an
    arc $\gamma$ in $T$ from $D$ to $D'$ through one of the involved
    intersections.  See the $t=0$ picture in
    Figure~\ref{figure:dual-torus-whitney-disk}.  Let $\mu$ and $\mu'$
    be meridians and $r$ and $s$ be the orders of $D$ and $D'$
    respectively.  Similarly to Case~1, we have $\mu \in G_{r+1}$ and
    $\mu'\in G_{s+1}$.  In addition, since $D''$ has order $r+s+1$,
    $\partial D'' \in G_{n-r-s}$ by
    Lemma~\ref{lemma:disk-boundary-in-lower-central-series}.

    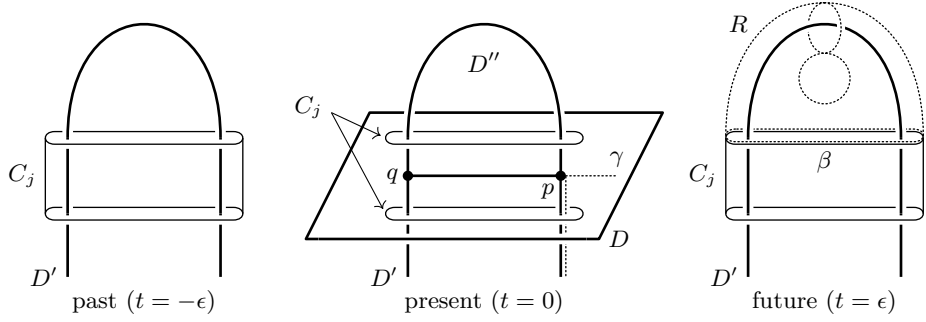
\begin{figure}[H]
      \leavevmode
      \hbox{\begin{tikzpicture}
          [x=5bp,y=5bp,scale=.95,line width=1pt,double distance=1pt,join=round,
          over/.style={draw=white,double=black,double distance=1pt,line width=1.5pt},
          midover/.style={draw=white,double=black,double distance=.5pt,line width=1.5pt},
          dotted/.style={line width=.5pt,densely dotted},
          hidden/.style={text=blue,opacity=1},
          hidden/.style={text=white,opacity=0}
          ]
          \small
          \clip (3,-6) rectangle (25,19);
          \def\c[#1]#2{coordinate(#2) node[hidden,#1]{\footnotesize$#2$}}
          
          \draw (8,-3) node[anchor=east]{$D'$} -- ++(0,5) \c[right]{q0};
          \draw (20,-3) -- ++(0,5) \c[right]{p0};

          \draw[midover] ($(q0)+(-1,.5)$) ..controls+(-1,0)and+(-1,0).. ++(0,-1) coordinate[midway](d0)
          -- ($(p0)+(1,-.5)$) ..controls+(1,0)and+(1,0).. ++(0,1) coordinate[midway](c0) -- cycle;
          
          \draw[over] (q0) -- ++(0,3) \c[right]{q} -- ++(0,3) \c[right]{q1};
          \draw[over] (p0) -- ++(0,3) \c[right]{p} -- ++(0,3) \c[right]{p1};

          \draw[midover] ($(q1)+(-1,.5)$) ..controls+(-1,0)and+(-1,0).. ++(0,-1) coordinate[midway](d1)
          -- ($(p1)+(1,-.5)$) ..controls+(1,0)and+(1,0).. ++(0,1) coordinate[midway](c1) -- cycle;

          \draw[line width=.5pt] (c0)--(c1) (d0)--(d1) node[midway,anchor=east]{$C_j$};

          \draw[over] (p1) ..controls+(0,12)and+(0,12).. (q1);
          \draw ($(p)!0.5!(q)$) ++(0,-10) node{past ($t=-\epsilon$)};
        \end{tikzpicture}}
      \hbox{\begin{tikzpicture}
          [x=5bp,y=5bp,scale=.95,line width=1pt,double distance=1pt,join=round,
          over/.style={draw=white,double=black,double distance=1pt,line width=1.5pt},
          midover/.style={draw=white,double=black,double distance=.5pt,line width=1.5pt},
          dotted/.style={line width=.5pt,densely dotted,cap=round},
          hidden/.style={text=blue,opacity=1},
          hidden/.style={text=white,opacity=0}
          ]
          \small
          \clip (-1,-6) rectangle (29,19);
          \def\c[#1]#2{coordinate(#2) node[hidden,#1]{\footnotesize$#2$}}
          \draw (1,0) \c[below]{A} -- (0,0) -- ++(5,10)
          -- ++(23,0) -- ++(-5,-10) node[anchor=west]{$D$} -- ++(-1,0) \c[below]{B};
          
          \draw (8,-3) node[anchor=east]{$D'$} -- ++(0,5) \c[right]{q0};
          \draw (20,-3) -- ++(0,5) \c[right]{p0};
          \draw[dotted] ($(p0)+(2pt,-2pt)$) -- ++($(0,-5)+(0,2pt)$);
          
          \draw[midover] ($(q0)+(-1,.5)$) ..controls+(-1,0)and+(-1,0).. ++(0,-1)
          coordinate[midway](d0) -- ($(p0)+(1,-.5)$) ..controls+(1,0)and+(1,0).. ++(0,1) -- cycle;
          
          \draw[over] (q0) -- ++(0,3) \c[right]{q} -- ++(0,3) \c[right]{q1};
          \draw[over] (p0) -- ++(0,3) \c[right]{p} -- ++(0,3) \c[right]{p1};
          \draw[draw=white,line width=3pt] ($(p0)+(2pt,0)$) -- ++(0,3);
          \draw[dotted] ($(p0)+(2pt,0)$) -- ($(p)+(2pt,0)$);

          \draw[midover] ($(q1)+(-1,.5)$) ..controls+(-1,0)and+(-1,0).. ++(0,-1)
          coordinate[midway](d1) -- ($(p1)+(1,-.5)$) ..controls+(1,0)and+(1,0).. ++(0,1) -- cycle;

          \draw[over] (p1) ..controls+(0,12)and+(0,12).. (q1);

          \draw[fill] (q) circle(1.5pt) node[anchor=east]{$q$} --
          (p) \c[right]{p} circle(1.5pt) node[anchor=north east,xshift=.3ex,yshift=.2ex)]{$p$};

          \draw[dotted] ($(p)+(2pt,0)$) - ++(4,0) node[anchor=south]{$\gamma$};
          \draw[over] (A)--(B);

          \draw[thin,shorten >=2pt,->] (2,10) node[anchor=east,yshift=2pt]{$C_j$} -- (d0);
          \draw[thin,shorten >=2pt,->] (2,10) -- (d1);
          \draw ($(p)!0.5!(q)$) ++(0,9) node{$D''$};
          \draw ($(p)!0.5!(q)$) ++(0,-10) node{present ($t=0$)};
        \end{tikzpicture}}
      \hbox{\begin{tikzpicture}
          [x=5bp,y=5bp,scale=.95,line width=1pt,double distance=1pt,join=round,
          over/.style={draw=white,double=black,double distance=1pt,line width=1.5pt},
          midover/.style={draw=white,double=black,double distance=.5pt,line width=1.5pt},
          dotted/.style={line width=.5pt,densely dotted,cap=round},
          hidden/.style={text=blue,opacity=1},
          hidden/.style={text=white,opacity=0}
          ]
          \small
          \clip (3,-6) rectangle (25,19);
          \def\c[#1]#2{coordinate(#2) node[hidden,#1]{\footnotesize$#2$}}
          
          \draw (8,-3) node[anchor=east]{$D'$} -- ++(0,5) \c[right]{q0};
          \draw (20,-3) -- ++(0,5) \c[right]{p0};

          \draw[midover] ($(q0)+(-1,.5)$) ..controls+(-1,0)and+(-1,0).. ++(0,-1) coordinate[midway](d0)
          -- ($(p0)+(1,-.5)$) ..controls+(1,0)and+(1,0).. ++(0,1) coordinate[midway](c0) -- cycle;
          
          \draw[over] (q0) -- ++(0,3) \c[right]{q} -- ++(0,3) \c[right]{q1};
          \draw[over] (p0) -- ++(0,3) \c[right]{p} -- ++(0,3) \c[right]{p1};

          \draw[midover] ($(q1)+(-1,.5)$) ..controls+(-1,0)and+(-1,0).. ++(0,-1) coordinate[midway](d1)
          -- ($(p1)+(1,-.5)$) node[midway,anchor=north,yshift=1pt]{$\beta$}
          ..controls+(1,0)and+(1,0).. ++(0,1) coordinate[midway](c1) -- cycle;
          
          \draw[dotted] ($(q1)+(-1,.5)+(0,1pt)$) ..controls+(-1,0)and+(-1,0).. ++(0,-1)
          coordinate[midway](d2)
          -- ($(p1)+(1,-.5)+(0,1pt)$) ..controls+(1,0)and+(1,0).. ++(0,1)
          coordinate[midway](c2) -- cycle;
          
          \draw[line width=.5pt] (c0)--(c1) (d0)--(d1) node[midway,anchor=east]{$C_j$};

          \draw[dotted] (c2) ..controls+(0,14)and+(0,14).. (d2) coordinate[midway](t)
          node[pos=.75,anchor=south east,yshift=-2pt]{$R$} (t) arc(90:270:1.3 and 2);

          \draw[over] (p1) ..controls+(0,12)and+(0,12).. (q1);

          \draw[draw=white,line width=3pt] (t) ++(1.3,-2) arc(0:30:1.3 and 2) arc(30:-30:1.3 and 2);

          \draw[dotted] (t) arc(90:-90:1.3 and 2) arc(90:450:2) ;

          \draw ($(p)!0.5!(q)$) ++(0,-10) node{future ($t=\epsilon$)};
        \end{tikzpicture}}

      \caption{The torus $C_j$ dual to the arc $\gamma$ passing
        through a paired intersection~$p$.}
      \label{figure:dual-torus-whitney-disk}
    \end{figure}

    Let $C_j$ be the torus illustrated in
    Figure~\ref{figure:dual-torus-whitney-disk}; $C_j$ is the union of
    two annuli in the $t=\pm\epsilon$ pictures, and additional two
    annuli connecting the boundary circles of the former annuli
    through $-\epsilon\le t\le\epsilon$; the connecting annuli are
    shown as two circles in the $t=0$ picture.  It is straightforward
    to see that $C_j$ is dual to the arc~$\gamma$, similarly to the
    Clifford torus in Case~1.

    Let $\beta$ be the circle shown in the $t=\epsilon$ picture;
    $\beta$ is the top boundary of the annulus part of $C_j$ in the
    $t=\epsilon$ picture.  The meridian $\mu$ of $D$ and the circle
    $\beta$ are standard basis curves of~$C_j$.  Since $\beta$ is the
    boundary of the punctured torus $R$ illustrated with dotted lines
    in the $t=\epsilon$ picture, and since $\mu'$ and $\partial D''$
    are (homotopic to) standard basis curves of $R$, we have
    $\beta=[\mu',\partial D'']$.  Since $\mu' \in G_{s+1}$ and
    $\partial D''\in G_{n-r-s}$, we have $\beta\in G_{n-r+1}$.  Since
    $\mu\in G_{r+1}$, this shows that the standard basis curves $\mu$
    and $\beta$ of the torus $C_j$ satisfy the promised property.
    \qedhere
  \end{case-named}
\end{proof}

\begin{proof}[Proof of Theorem~\ref{theorem:milnor-theorem-rational-tower}]
  Let $G=\pi_1(X\sm T)$ where $T$ is a twisted Whitney tower of
  order $n$ in a rational homology 4-ball~$X$ bounded by an
  $m$-component link $L\subset \partial X$.  By
  Lemma~\ref{lemma:tower-exterior-homology}~(1), a given meridian map
  $F\to G$ induces an isomorphism
  $\Q^m\cong H_1(F;\Q) \cong H_1(G;\Q)$.

  By Lemma~\ref{lemma:tower-exterior-homology}~(2), $H_2(X\sm T;\Q)$
  is generated by classes $[C]$ of tori $C$ with standard basis curves
  $\alpha$, $\beta$ such that $\alpha\in G_k$ and $\beta\in G_{n-k+2}$
  for some~$k$.  By a standard argument (e.g., see \cite[(proofs of)
  Lemma~2.3 and Lemma~2.1]{Freedman-Teichner:1995-2}), such a toral
  class $[C]$ is contained in the kernel of
  $H_2(X\sm T) \to H_2(G/G_{n+1})$.  Since
  $H_2(X\sm T;\Q) \to H_2(G;\Q)$ is surjective, it follows that
  $H_2(G;\Q) \to H_2(G/G_{n+1};\Q)$ is zero.
  
  We now invoke Stallings-Dwyer theorem for rational
  coefficients~\cite{Stallings:1965-1, Dwyer:1975-1}: \emph{if a group
    homomorphism $\Gamma\to G$ induces an isomorphism on
    $H_1(\Gamma;\Q) \cong H_1(G;\Q)$ and an epimorphism
    \[
      H_2(\Gamma;\Q) \to H_2(G;\Q)/\Ker\{H_2(G;\Q) \rightarrow
      H_2(G/G_{k};\Q)\},
    \]
    then it induces an isomorphism}
    \[
      (\Gamma^\Q_{k}/\Gamma^\Q_{k+1}) \otimesover{\Z} \Q
      \xrightarrow{\cong} (G^\Q_k/G^\Q_{k+1}) \otimesover{\Z} \Q.
    \]

  Applying this to the meridian map $F\to G$, we obtain an
  isomorphism
  \[
    (F^\Q_{k+1}/F^\Q_{k+2})\otimesover{\Z}\Q \cong
    (G^\Q_{k+1}/G^\Q_{k+2}) \otimesover{\Z}\Q
  \]
  for each~$k\le n$.

  Therefore, to complete the proof of
  Theorem~\ref{theorem:milnor-theorem-rational-tower}, it suffices to
  show that $F^\Q_{k} = F_{k}$ for all~$k$.  It is straightforward to
  verify this by an induction: $F^\Q_1 = F = F_1$, and if
  $F^\Q_{k} = F_{k}$, then
  $F^\Q_k/[F,F^\Q_k] = F_k/F_{k+1}\cong \sL_k$ is a finitely
  generated free abelian group (e.g., by the Hall basis theorem), and
  so by definition, we have
  \[
    F^\Q_{k+1} = \Ker\{F_k \rightarrow (F_k/F_{k+1})\otimesover{\Z} \Q
    \} = \Ker\{F_k \rightarrow F_k/F_{k+1} \} = F_{k+1}.  \qedhere
  \]
\end{proof}

\subsection{Whitney towers and Milnor invariants}
\label{section:proof-vanishing-milnor-invariant}

Now we are ready to prove the main result of this section.

\begin{proof}[Proof of Theorem~\ref{theorem:whitney-tower-milnor-invariant}]
  Suppose $L$ is a framed link in $S^3$, $X$ is a rational homology
  4-ball with $\partial X=S^3$, and $T$ is a twisted Whitney tower of
  order $n$ in $X$ bounded by~$L$.  We will prove that $\mu_k(L)=0$
  for $k< n$ and $\mu_n(L) = \eta_n (t^\tw_n(T))$.

  Let $\pi=\pi_1(S^3\sm L)$, $G=\pi_1(X\sm T)$, and let
  $\lambda_i\in \pi$ be a pushoff of the $i$th component of~$L$ taken
  along the given framing.  (By
  Remark~\ref{remark:links-bounding-towers-in-QHD4}, $\lambda_i$ is a
  zero linking longitude if $n>0$.)  By
  Lemma~\ref{lemma:order-zero-disk-boundary}, the image of $\lambda_i$
  lies in $G_{n+1}$.

  We proceed inductively.  Suppose $k\le n$ and $\mu_{k-1}(L)$ has
  been shown to vanish.  (We assume nothing for $k=0$.)  Let $F$ be
  the free group of the same rank as the number of components of~$L$.
  By Milnor's theorem~\cite[Theorem~4]{Milnor:1957-1} (see
  Section~\ref{subsection:review-milnor-invariant}) and by
  Theorem~\ref{theorem:milnor-theorem-rational-tower}, we obtain the
  following commutative diagram with vertical arrows isomorphisms:
  \[
    \begin{tikzcd}
      \pi_{k+1}/\pi_{k+2} \ar[r] \ar[d, "\cong" left]&
      G_{k+1}/G_{k+2} \ar[r] &
      (G^\Q_{k+1}/G^\Q_{k+2}) \otimes \Q \ar[d, "\cong" right]
      \\
      F_{k+1}/F_{k+2} \ar[rr] &&
      (F_{k+1}/F_{k+2}) \otimes \Q      
    \end{tikzcd}
  \]
  Let $w_i\in F_{k+1}/F_{k+2}$ be the image of~$\lambda_i$. Then by
  definition,
  $\mu_k(L) = \sum_i X_i \otimes w_i \in \sL_1\otimes
  (F_{k+1}/F_{k+2}) \cong \sL_1\otimes \sL_{k+1}$.

  If $k\le n-1$, then since $\lambda_i$ is sent into
  $G_{n+1} \subset G_{k+2}$, it follows that the image of $\lambda_i$
  in $(F_{k+1}/F_{k+2}) \otimes \Q$ is trivial.  The bottom arrow of
  the diagram is a monomorphism since $F_{k+1}/F_{k+2}$ is torsion
  free abelian.  It follows that $w_i \in F_{k+1}/F_{k+2}$ is trivial.
  Therefore $\mu_k(L)=0$.

  If $k=n$, then the image of $\lambda_i$ in $G_{n+1}/G_{n+2}$ is
  given by Lemma~\ref{lemma:order-zero-disk-boundary}.  By comparing
  the longitude formula in Lemma~\ref{lemma:order-zero-disk-boundary}
  and the defining formula of $\eta_n$ in
  Definition~\ref{definition:tree-invariant-of-twisted-tower}, it
  follows that
  \[
    \mu_n(L)\otimes 1 = \eta_n(t^\tw_n(T)) \otimes 1 \in \sL_1
    \otimes (F_{n+1}/F_{n+2}) \otimes \Q = \sL_1 \otimes \sL_{n+1}
    \otimes \Q.
  \]
  Since $\sL_1\otimes\sL_{n+1} \to \sL_1\otimes\sL_{n+1}\otimes\Q$ is
  injective, $\mu_n(L) = \eta_n(t^\tw_n(T))$ in
  $\sL_1\otimes\sL_{n+1}$.
\end{proof}

As a consequence of
Theorem~\ref{theorem:whitney-tower-milnor-invariant}, we prove that
Milnor invariants are preserved under rational Whitney tower
concordance.  It will be used in
Section~\ref{section:framed-classification}.

\begin{corollary}
  \label{corollary:whitney-concordance-invariance-milnor-invariant}
  Suppose two framed links $L$ and $L'$ in $S^3$ are order $n+1$
  twisted Whitney tower concordant in a rational homology
  $S^3\times I$ $(n\ge 0)$.  If $L$ bounds a twisted Whitney tower of
  order $n$ in a rational homology 4-ball, then so does $L'$, and
  furthermore $\mu_n(L)=\mu_n(L')$.
\end{corollary}

\begin{remark}
  In Section~\ref{subsection:rational-whitney-filtration}, we will
  show that a link $L$ bounds a twisted Whitney tower of order $n$ in
  a rational homology 4-ball if and only if $\mu_k(L)=0$ for $k<n$.
  See Theorem~\ref{theorem:rational-triviality-characterization}.
\end{remark}

\begin{proof}[Proof of Corollary~\ref{corollary:whitney-concordance-invariance-milnor-invariant}]
  Let $T$ be a twisted Whitney tower of order $n$ bounded by $L$ in a
  rational homology 4-ball, and let $C$ be an order $n+1$ twisted
  Whitney tower concordance between $L$ and $L'$ in a rational
  homology $S^3\times I$.  Stacking $T$ and $C$, we obtain an order
  $n+1$ twisted Whitney tower $T'$ bounded by~$L'$ in another rational
  homology 4-ball.  By
  Theorem~\ref{theorem:whitney-tower-milnor-invariant},
  $\mu_k(L)=0=\mu_k(L')$ for $k<n$ and thus $\mu_n(L)$ and $\mu_n(L')$
  are well-defined.  Since all order $n$ intersections of $C$ are
  paired up by Whitney disks, we have $t^\tw_n(T) = t^\tw_n(T')$.  By
  Theorem~\ref{theorem:whitney-tower-milnor-invariant}, it follows
  that $\mu_n(L)=\mu_n(L')$.
\end{proof}

\section{Links and Whitney towers in rational homology 4-space}
\label{section:rational-theory-to-integral-theory}

In what follows we fix the number $m$ of components of links.  As in
the introduction, define $\ol\W^\tw_n$ to be the set of framed
$m$-component links $L$ in $S^3$ bounding a twisted Whitney tower $T$
of order $n$ in a rational homology 4-ball with boundary~$S^3$.
(Recall that $\ol\W_0^\tw$ is the set of all links in~$S^3$ by
Remark~\ref{remark:links-bounding-towers-in-QHD4}.)  Let
$\ol\sW^\tw_n$ to be the set of equivalence classes of links in
$\ol\W^\tw_n$ under order $n+1$ twisted Whitney tower concordance in a
rational homology $S^3\times I$.  (Readers may find that this is
different from the defining condition in the introduction, but we will
show that they are equivalent in
Section~\ref{subsection:twisted-graded-quotients},
Corollary~\ref{corollary:whitney-concordance-vs-band-sum}.)

In this section, we will show that Milnor invariants characterize
links in $\ol\W^\tw_n$.  Using this we will show that $\ol\sW^\tw_n$
is an abelian group under band sum, and compute the structure of the
abelian group~$\ol\sW^\tw_n$.


\subsection{Some results of the integral twisted theory}
\label{section:review-integral-theory}

Our approach relies in an essential way on the work of Conant,
Schneiderman and Teichner on Whitney tower concordance in
$S^3\times I$~\cite{Conant-Schneiderman-Teichner:2012-2,
  Conant-Schneiderman-Teichner:2012-3,
  Conant-Schneiderman-Teichner:2012-1,
  Conant-Schneiderman-Teichner:2012-4}.  In this subsection we quickly
review parts of their work we need, focusing on the twisted case, and
setup notations.

Let $\W_n^\tw$ be the set of $m$-component framed links in $S^3$
bounding a twisted Whitney tower of order $n$ in~$D^4$.
Let $\sW^\tw_n$ be the set of order $n+1$ twisted Whitney tower
concordance classes of links in~$\W^\tw_n$.  Then the band sum of two
classes is well defined on $\sW^\tw_n$, independent of the choice of
representative links and the choice of
bands~\cite[Lemma~3.4]{Conant-Schneiderman-Teichner:2012-2}.  The set
$\sW^\tw_n$ is an abelian group under band sum.  In particular, for
$L, L' \in \W^\tw_n$, $[L]=[L']$ in $\sW^\tw_n$ if and only if
$L\csumover{\beta}-L'$ is in~$\W^\tw_{n+1}$ for some~$\beta$.  Often
we write $\sW^\tw_n = \W^\tw_n / \W^\tw_{n+1}$.

For a twisted Whitney tower $T$ of order $n$, they define an invariant
$\tau^\tw_n(T) \in \cT^\tw_n$, which is the class of the formal sum
$t^\tw_n(T)$ described in
Definition~\ref{definition:tree-invariant-of-twisted-tower}, in a
certain quotient $\cT^\tw_n$ of the free abelian group generated by
order $n$ trees and order $\frac n2$
$\tw$-trees~\cite{Conant-Schneiderman-Teichner:2012-2}.  We do not
need the precise definition of~$\cT^\tw_n$.  A key feature we will use
is the following:

\begin{theorem}[Order Raising~\cite{Conant-Schneiderman-Teichner:2012-2}]
  \label{theorem:cst-order-rasing}
  If $L$ bounds a twisted Whitney tower $T$ of order $n$ in $D^4$ with
  $\tau^\tw_n(T)=0$ in $\cT^\tw_n$, then $L$ bounds a twisted Whitney
  tower of order $n+1$ in~$D^4$.
\end{theorem}

Any $\theta\in \cT^\tw_n$ is realized by a link in the following
sense: there is an epimorphism
$R^\tw_n\colon \cT^\tw_n \to \sW^\tw_n$, called the \emph{realization
  map}, such that $R^\tw_n(\theta)$ is the class of a link bounding an
order $n$ twisted Whitney tower $T$ in $D^4$ with
$\tau^\tw_n(T)=\theta$ \cite{Conant-Schneiderman-Teichner:2012-2}.
(This condition determines the class $R^\tw_n(\theta)\in \sW^\tw_n$
uniquely by Theorem~\ref{theorem:cst-order-rasing}.)

The summation $\eta_n$ described in
Definition~\ref{definition:summation-map} induces a homomorphism
$\eta_n\colon \cT^\tw_n \to \sD_n$ (e.g. see
\cite[Section~4.3]{Conant-Schneiderman-Teichner:2012-3}).  Also, the
Milnor invariant of order $n$ gives rise to a homomorphism
$\mu_n\colon \sW^\tw_n \to
\sD_n$~\cite{Conant-Schneiderman-Teichner:2012-3}.  We state some
necessary facts as a theorem.

\begin{theorem}[Conant-Schneiderman-Teichner
  \cite{Conant-Schneiderman-Teichner:2012-2,
    Conant-Schneiderman-Teichner:2012-3,
    Conant-Schneiderman-Teichner:2012-1,
    Conant-Schneiderman-Teichner:2012-4}]
  \label{theorem:cst-integral-theory}
  \leavevmode\Nopagebreak
  \begin{enumerate}
  \item\label{item:kernel-of-mu} For $n\not\equiv 2 \mod 4$,
    $\eta_n\colon \cT^\tw_n \to \sD_n$ and
    $\mu_n\colon \sW^\tw_n\to \sD_n$ are isomorphisms.  For $n=4k-2$,
    $\eta_{4k-2}\colon \cT^\tw_{4k-2} \to \sD_{4k-2}$ is an
    epimorphism with kernel isomorphic to $\Z_2\otimes \sL_{k}$.
  \item\label{item:realization-of-kernel-of-mu} For each $\theta$ in
    $\Ker\{\eta_{4k-2}\colon \cT^\tw_{4k-2} \to \sD_{4k-2}\}$,
    $R_n^\tw(\theta)\in \sW^\tw_{4k-2}$ is the class of a link
    obtained by starting with the figure eight knot, applying Bing
    doubling to certain components repeatedly, and then applying
    internal band sum operations connecting distinct components.
  \end{enumerate}
\end{theorem}

The main conjecture is that $R^\tw_n$ is an isomorphism
$\cT^\tw_n\cong\sW^\tw_n$ for $n\equiv 2 \mod 4$.  This is equivalent
to the higher order Arf invariant
conjecture~\cite{Conant-Schneiderman-Teichner:2012-2}.


\subsection{Rational twisted Whitney tower filtration}
\label{subsection:rational-whitney-filtration}


In our characterization of links in $\ol\W^\tw_n$, the following
straightforward observation based on earlier known facts is essential.
Let
$\sB_{4k-2} := \Ker\{ \eta_{4k-2}\colon \cT^\tw_{4k-2} \to
\sD_{4k-2}\}$.  As stated in
Theorem~\ref{theorem:cst-integral-theory}~(\ref{item:kernel-of-mu}),
$\sB_{4k-2} \cong \Z_2\otimes \sL_{k}$.  We say that a link is
\emph{rationally slice} if it bounds slicing disks in a rational
homology 4-ball.  When $R$ is a ring, a link is \emph{$R$-slice} if it
bounds slicing disks in an $R$-homology 4-ball.

\begin{lemma}
  \label{lemma:kernel-realization-by-rationally-slice-links}
  For any $\theta\in \sB_{4k-2}$, the realization
  $R^\tw_{4k-2}(\theta) \in \sW^\tw_{4k-2}$ is represented by a
  $\Z[\frac12]$-slice link $L(\theta) \in \W^\tw_{4k-2}$.
\end{lemma}

\begin{proof}
  The figure eight knot bounds a slice disk in a rational
  $\Z[\frac12]$-homology 4-ball, by \cite[Proof of Theorem~4.16,
  Figure~6]{Cha:2003-1}.  If a link bounds slice disks in a
  4-manifold, both Bing doubling operation on a component and internal
  band sum operation joining distinct components give another link
  bounding slice disks in the same 4-manifold.  From this and
  Theorem~\ref{theorem:cst-integral-theory}~(\ref{item:realization-of-kernel-of-mu}),
  the conclusion stated above follows.
\end{proof}

For brevity, we write
$\sB_n := \Ker\{ \eta_{n}\colon \cT^\tw_{n} \to \sD_{n}\}$ for
any~$n$; for $n\not\equiv 2\mod 4$, $\sB_n=0$ by
Theorem~\ref{theorem:cst-integral-theory}~(\ref{item:kernel-of-mu}),
and thus
Lemma~\ref{lemma:kernel-realization-by-rationally-slice-links} is
vacuously true.

For two $m$-component links $L$ and $L'$ in $S^3$, we denote by
$L\csumover{\beta} L'$ their band sum defined using a collection
$\beta$ of $m$ bands joining the $i$th component of $L$ and that of
$L'$ in the split union $L\sqcup L'$.  That is, $L\csumover{\beta} L'$
is the result of $m$ internal band sum operations (ambient surgery) on
$L\sqcup L'$.

We will often use that there is a standard genus zero cobordism in
$S^3\times I$ between $(L\sqcup L')\times 0 \subset S^3\times 0$ and
$(L\csumover{\beta}L')\times 1 \subset S^3\times 1$:
\[
  \textstyle \big((L\sqcup L')\times [0,\frac12]\big) \cup
  (\beta\times\frac12) \cup
  \big((L\csumover{\beta}L')\times[\frac12,1]\big).
\]


\begin{lemma}
  \label{lemma:rational-vs-integral-towers-for-link}
  Suppose $L$ is a link bounding a twisted Whitney tower of order $n$
  in~$D^4$.  Then the following are equivalent:
  \begin{enumerate}
  \item There is $\theta \in \sB_n$ such that a band sum
    $L\csumover{\beta} L(\theta)$ bounds a twisted Whitney tower of
    order $n+1$ in~$D^4$ for any~$\beta$.  Here $L(\theta)$ is the
    link in
    Lemma~\ref{lemma:kernel-realization-by-rationally-slice-links}.
  \item $L$ bounds a twisted Whitney tower of order $n+1$ in a
    $\Z[\frac12]$-homology 4-ball.
  \item $L$ bounds a twisted Whitney tower of order $n+1$ in a
    rational homology 4-ball.
  \item $\mu_n(L)=0$.
  \end{enumerate}
\end{lemma}

Note that for $n\not\equiv 2\mod 4$, (1) is equivalent to that $L$
bounds a twisted Whitney tower of order $n+1$ in~$D^4$.

\begin{proof}
  Suppose $L\csumover{\beta}L(\theta)$ bounds an order $n+1$ twisted
  Whitney tower $T$ in $D^4$ as in~(1).  Then a standard argument
  gives an order $n+1$ twisted Whitney tower concordance in
  $S^3\times I$, say $T'$, between $L$ and~$L(\theta)$.  Details are
  as follows: first attach to $T$ a standard genus zero cobordism
  between $(L\csumover{\beta}L(\theta)) \times 1$ and the split union
  $(L \sqcup L(\theta)) \times 0$ in $S^3\times I$.  This gives a
  tower $T''$ in $D^4$ bounded by $L \sqcup L(\theta)$.  Identify
  $D^4$ with $\ol{S^3 \sm D^3}\times I$ in such a way that $L$ and
  $L(\theta)$ lie in the first and second summands of
  $\partial(\ol{S^3 \sm D^3}\times I)= \ol{S^3 \sm D^3}
  \cupover{\partial} -\ol{S^3 \sm D^3} = S^3\csum -S^3$ respectively.
  Then the promised $T'\subset S^3\times I$ is the image of $T''$
  under the inclusion $\ol{S^3 \sm D^3}\times I \subset S^3\times I$.
  
  Attach to $T'$ a slicing disk of $-L(\theta)$ in a
  $\Z[\frac12]$-homology 4-ball, which exists by
  Lemma~\ref{lemma:kernel-realization-by-rationally-slice-links}.  The
  result is a twisted Whitney tower of order $n+1$ in a
  $\Z[\frac12]$-homology 4-ball which is bounded by~$L$.  This shows
  (1)~$\Rightarrow$~(2).

  (2)~$\Rightarrow$~(3) is trivial.  (3)~$\Rightarrow$~(4) is an
  immediate consequence of
  Theorem~\ref{theorem:whitney-tower-milnor-invariant}.
  
  Suppose (4) holds.  Choose a twisted Whitney tower $T$ of order $n$
  in $D^4$ bounded by~$L$.  Then $\eta_n(\tau^\tw_n(T)) = \mu_n(L)=0$
  by using Theorem~\ref{theorem:whitney-tower-milnor-invariant} (or
  the original integral
  version~\cite[Theorem~6]{Conant-Schneiderman-Teichner:2012-3}).
  Therefore $\tau^\tw_n(T)\in \sB_{4k-2}$.  Let
  $\theta:=-\tau^\tw_n(T)$.  Then $L(\theta)$ bounds a twisted Whitney
  tower $T''$ in $D^4$ with $\tau^\tw_n(T'')=-\tau^\tw_n(T)$.  Attach
  the disjoint union of $T$ and $T''$ to a standard genus zero
  cobordism between $L\csumover\beta L(\theta)$ and
  $L\sqcup L(\theta)$, to obtain an order $n$ twisted Whitney tower
  with $\tau^\tw_n = \tau^\tw_n(T)+\tau^\tw_n(T'') = 0$.  By
  Theorem~\ref{theorem:cst-order-rasing}, it follows that
  $L\csumover{\beta}L(\theta)$ lies in~$\W^\tw_{n+1}$.  This shows
  (4)~$\Rightarrow$~(1).
\end{proof}

We will use connected sum of links as a special case of band sum.  A
precise description is as follows.  Let $L$ be a link with $m$
components in~$S^3$.  Fix $m$ distinct interior points
$z_1,\ldots,z_m \in D^2$.  Choose an embedding
$b\colon D^2\times I \to S^3$ such that the inverse image of the $i$th
component of $L$ under $b$ is equal, as an oriented arc,
to~$z_i\times I$.  We call $b$ a \emph{basing} for~$L$.  Let $L'$ be
another $m$-component link with a basing~$b'$.  Let
$Y=\ol{S^3\sm b(D^2\times I)}$ and $Y'=\ol{S^3\sm b'(D^2\times I)}$.
Define the \emph{connected sum} $L\csumover{(b,b')} L'$ by
$(S^3,L\csumover{(b,b')} L') = (Y,L\cap Y) \cupover{\partial}
(Y',L\cap Y')$ where $\partial Y$ is identified with $\partial Y'$
under $b(z,t) \mapsto b'(z,1-t)$, $(z,t)\in \partial(D^2\times I)$.
That is, the connected sum $L\csumover{(b,b')} L'$ is the band sum
defined using the pair of basings $(b,b')$ as bands.

Now we are ready to present a complete characterization of links
bounding a twisted Whitney tower of a given order in a rational and
$\Z[\frac12]$-homology 4-ball.

\begin{theorem}
  \label{theorem:rational-triviality-characterization}
  For any link $L$ in $S^3$ and $n\ge 0$, the following are
  equivalent:
  \begin{enumerate}
  \item $L \in \ol\W^\tw_{n+1}$, that is, $L$ bounds a twisted
    Whitney tower of order $n+1$ in a rational homology 4-ball.
  \item $\mu_k(L)=0$ for $k\le n$.
  \item For any basing $b$ for $L$, there is a rationally slice link
    $L_0$ with a basing $b_0$ such that
    $L\csumover{(b,b_0)}L_0 \in \W^\tw_{n+1}$.
  \item For any basing $b$ for $L$, there is a $\Z[\frac12]$-slice
    link $L_0$ with a basing $b_0$ such that
    $L\csumover{(b,b_0)}L_0 \in \W^\tw_{n+1}$.
  \item $L$ bounds a twisted Whitney tower of order $n+1$ in a
    $\Z[\frac12]$-homology 4-ball.
  \end{enumerate}
\end{theorem}

From Theorem~\ref{theorem:rational-triviality-characterization}, the
twisted case of Theorem~\ref{theorem:general-coefficients-intro} in
the introduction follows immediately: \emph{for any subring $R$ of
  $\Q$ containing $\tfrac12$, a link in $S^3$ bounds a twisted Whitney
  tower of order $n$ in an $R$-homology 4-ball if and only if the link
  bounds a twisted Whitney tower of order $n$ in a rational homology
  4-ball.}

Also, Theorem~\ref{theorem:geometric-characterization-milnor-intro} in
the introduction is exactly the equivalence of (1) and (2) in
Theorem~\ref{theorem:rational-triviality-characterization}.

\begin{proof}
  [Proof of
  Theorem~\ref{theorem:rational-triviality-characterization}]

  (1)~$\Rightarrow$~(2) is
  Theorem~\ref{theorem:whitney-tower-milnor-invariant}~(1).  Suppose
  (2) holds.  Since $L\in \W^\tw_0$ and $\mu_0(L)=0$, there is
  $\theta_0 \in \sB_0$ and a basing $b_0$ for $L(\theta_0)$ such that
  $L\csumover{(b,b_0)}L(\theta_0) \in \W^\tw_1$ for any $b$ for~$L$,
  by Lemma~\ref{lemma:rational-vs-integral-towers-for-link}.
  Choose a basing $b_0'$ for $L(\theta_0)$ which is disjoint
  from~$b_0$.  If $n\ge 1$, then by the same argument, using
  $\mu_1(L)=0$, there is $\theta_1\in \sB_1$ and two disjoint basings
  $b_1$ and $b_1'$ for $L(\theta_1)$ such that
  $(L\csumover{(b,b_0)}L(\theta_0))\csumover{(b_0',b_1)}L(\theta_1)
  \in \W^\tw_2$.  Repeating this, choose $\theta_i\in \sB_i$ and
  disjoint basings $b_i$ and $b_i'$ for $L(\theta_i)$ for
  $i=0,\ldots,n$ such that
  \[
    \Big(\cdots \big(L\csumover{(b,b_0)}L(\theta_0)\big)
    \csumover{(b_0',b_1)} \cdots \;\;\Big) \csumover{(b_{n-1}',b_{n})}
    L(\theta_n) \in \W^\tw_{n+1}.
  \]
  Since the basings are disjoint, the above connected sum operations
  are associative.  It follows that
  $L_0 := L(\theta_0) \csumover{(b_0',b_1)} \cdots
  \csumover{(b_{n-1}',b_{n})} L(\theta_n)$ with the basing $b_0$
  satisfies (4).  This shows (2)~$\Rightarrow$~(4).
  (4)~$\Rightarrow$~(3) is straightforward.

  Both (3)~$\Rightarrow$~(1) and (4)~$\Rightarrow$~(5) are are shown
  by the standard argument used in the proof of (1)~$\Rightarrow$~(2)
  of Lemma~\ref{lemma:rational-vs-integral-towers-for-link}, using
  that $L_0$ is rationally slice and $\Z[\frac12]$-slice respectively.
  (5)~$\Rightarrow$~(1) is trivial.  The completes the proof.
\end{proof}

\subsection{Rational twisted graded quotient}
\label{subsection:twisted-graded-quotients}

For brevity, write $L\sim L'$ if two framed links $L$ and $L'$ in
$\ol\W^\tw_n$ are order $n+1$ twisted Whitney tower concordant in a
rational homology~$S^3\times I$.  Recall that
$\ol\sW^\tw_n=\ol\W^\tw_n/\mathop{\sim}$.

\begin{theorem}
  \label{theorem:computation-of-twisted-graded-quotient}
  \leavevmode\Nopagebreak
  \begin{enumerate}
  \item $\ol\sW^\tw_n$ is an abelian group under band sum
    $[L]+[L']=[L\csum_{\beta}L']$.
  \item $\mu_n\colon \ol\sW^\tw_n \to \sD_n$ is a group
    isomorphism.
  \item For the $m$-component case, $\ol\sW^\tw_n$ is a free abelian
    group of rank $\cM(m,n)$, where $\cM(m,k)$ is the number defined
    in Remark~\ref{remark:rank-of-D_n}.
  \item $\sW^\tw_n \to \ol\sW^\tw_n$ is an epimorphism with
    kernel equal to
    $\sK^\tw_n:=\Ker\{\mu_n\colon \sW^\tw_n\to\sD_n\}$.
    $\sW^\tw_n \cong \ol\sW^\tw_n$ for $n\not\equiv 2\mod 4$.
  \end{enumerate}
\end{theorem}

In the following proof, we will use Krushkal's
additivity~\cite{Krushkal:1998-1}: \emph{if $\mu_q(L)=0=\mu_q(L')$ for
  $q<n$, then $\mu_n(L\csumover\beta L') = \mu_n(L)+\mu_n(L')$ for any
  bands~$\beta$.}  It can also be seen by using
Theorem~\ref{theorem:whitney-tower-milnor-invariant}.

\begin{proof}
  [Proof of Theorem~\ref{theorem:computation-of-twisted-graded-quotient}]
  
  We first claim that for $L, L'\in \ol\W^\tw_n$, $L\sim L'$ if and
  only if $\mu_n(L)=\mu_n(L')$.  The only if direction is true by
  Corollary~\ref{corollary:whitney-concordance-invariance-milnor-invariant}.
  Conversely, if $\mu_n(L)=\mu_n(L')$, then for any choice of bands
  $\beta$, $\mu(L\csum_{\beta}-L') = \mu(L)-\mu(L') = 0$ by the
  additivity.  By
  Theorem~\ref{theorem:rational-triviality-characterization}, it
  follows that $L\csum_{\beta} -L'\in \ol\W^\tw_{n+1}$.  It implies
  $L\sim L'$ by the standard argument for band sum which was used in
  the proof of (1)~$\Rightarrow$~(2) of
  Lemma~\ref{lemma:rational-vs-integral-towers-for-link}.  This
  completes the proof of the claim.

  By the claim, $\mu_n\colon \ol\sW^\tw_n \to \sD_n$ is an
  injective \emph{function}.  Since the diagram
  \begin{equation}
    \label{equation:mu-commutativity}
    \vcenter{\hbox{\begin{tikzcd}
          \sW^\tw_n \ar[rr] \ar[rd, two heads, "\mu_n" below left] & &
          \ol\sW^\tw_n \ar[ld, "\mu_n"]
          \\
          & \sD_n
        \end{tikzcd}}}
  \end{equation}
  is commutative and since $\mu_n\colon \sW^\tw_n \to \sD_n$ is
  surjective by
  Theorem~\ref{theorem:cst-integral-theory}~(\ref{item:kernel-of-mu}),
  it follows that $\mu_n\colon \ol\sW^\tw_n \to \sD_n$ is
  surjective.  Therefore $\mu_n\colon \ol\sW^\tw_n \to \sD_n$ is
  bijective.

  Also, from the claim, it follows that the class
  $[L\csumover{\beta}L']$ of a band sum is determined by the classes
  $[L]$ and $[L]\in \ol\sW^\tw_n$, independent of the choice of
  $\beta$, since $\mu_n(L\csumover{\beta}L') = \mu_n(L)+\mu_n(L')$ is
  determined by $\mu_n(L)$ and~$\mu_n(L')$.  Thus
  $[L]+[L']=[L\csumover{\beta}L']$ is a well defined operation
  on~$\ol\sW^\tw_n$.

  Since $\sD_n$ is a group and
  $\mu_n\colon \ol\sW^\tw_n \to \sD_n$ is a bijective function
  preserving the addition, $\ol\sW^\tw_n$ is a group under the
  addition and $\mu_n\colon \ol\sW^\tw_n \to \sD_n$ is a group
  isomorphism.  This proves (1) and~(2).

  Since $\sD_n$ is a free abelian group of rank $\cM(m,n)$ (see
  Remark~\ref{remark:rank-of-D_n}), so is~$\ol\sW^\tw_n$.  This
  shows~(3).

  Since $\mu_n\colon \ol\sW^\tw_n \to \sD_n$ is an isomorphism,
  from \eqref{equation:mu-commutativity} it follows that
  $\sW^\tw_n \to \ol\sW^\tw_n$ is surjective and has kernel
  $\sK^\tw_n:=\Ker\{\mu_n\colon \sW^\tw_n\to\sD_n\}$.  By
  Theorem~\ref{theorem:cst-integral-theory}~(\ref{item:kernel-of-mu}),
  $\sK^\tw_n=0$ for $n\not\equiv 2\mod 4$.  This shows~(4).
\end{proof}

Since $[L]=0$ in $\ol\sW^\tw_n$ if and only if
$L \in \ol\W^\tw_{n+1}$, the following is a direct consequence of
Theorem~\ref{theorem:computation-of-twisted-graded-quotient}~(1).

\begin{corollary}
  \label{corollary:whitney-concordance-vs-band-sum}
  $[L] = [L']$ in $\ol\sW^\tw_n$ if and only if
  $L\csumover{\beta}-L' \in \ol\W^\tw_{n+1}$
\end{corollary}

We conclude this section with a discussion on the higher order Arf
invariants.  Recall that
$\sB_{4k-2} = \Ker\{\eta_{4k-2}\colon \cT^\tw_{4k-2} \rightarrow
\sD_{4k-2}\} \cong \Z_2\otimes \sL_k$ and
$\sK^\tw_{4k-2} = \Ker\{\mu_{4k-2}\colon \sW^\tw_{4k-2} \rightarrow
\sD_{4k-2}\}$.  In \cite{Conant-Schneiderman-Teichner:2012-2}, Conant,
Schneiderman and Teichner showed that
$R^\tw_{4k-2}\colon \cT^\tw_{4k-2} \to \sW^\tw_{4k-2}$ restricts to an
epimorphism
$\alpha^\tw_k\colon \sB_{4k-2} \twoheadrightarrow \sK^\tw_{4k-2}$.  They
defined the \emph{$k$th higher order Arf invariant} by
\[
  \Arf_k:=(\alpha^\tw_k)^{-1} \colon \sK^\tw_{4k-2}
  \xrightarrow{\cong} \sB_{4k-2}/\Ker \alpha_k.
\]
The higher order Arf invariant conjecture asserts that $\alpha^\tw_k$
is an isomorphism. In particular, it claims that $\Arf_k$ is not
identically trivial.

Using the definition of $\Arf_k$, it is straightforward to reformulate
Theorem~\ref{theorem:computation-of-twisted-graded-quotient}~(4) to
the following statement:

\begin{corollary}
  \label{corollary:higher-order-arf-graded-quotient}
  The epimorphism $\sW^\tw_n \to \ol\sW^\tw_n$ is an isomorphism
  if and only if either $n \not\equiv 2\mod 4$, or $n=4k-2$ and
  $\Arf_k \equiv 0$.
\end{corollary}

Theorem~\ref{theorem:higher-order-arf-rational-tower-intro} in the
introduction is an immediate consequence of
Corollary~\ref{corollary:higher-order-arf-graded-quotient}.

\section{Framed classification}
\label{section:framed-classification}

Let $\ol\W_n$ be the set of framed links in $S^3$ which bound an order
$n$ framed Whitney tower in a rational homology 4-ball.  The goal of
this section is to understand the structure of the filtration
$\{\ol\W_n\}$ and its graded quotients $\ol\sW_n$ which is a framed
analog of~$\ol\sW^\tw_n$.  We will define $\ol\sW_n$ precisely in
Section~\ref{subsection:framed-graded-quotients}.  The main result is
as follows.

\begin{theorem}
  \label{theorem:structure-graded-quotient-framed}
  For the $m$-component case, the following hold.
  \begin{enumerate}
  \item The Milnor invariant of order $n$ gives rise to an epimorphism
    $\mu_n\colon \ol\sW_n \to \Z^{\cM(m,n)}$ onto a free abelian
    group of rank~$\cM(m,n)$.
  \item If $n$ is even, then $\mu_n$ is an isomorphism
    $\ol\sW_n \cong \Z^{\cM(m,n)}$.
  \item If $n=2\ell-1$, there is a short exact sequence
    \[
      0 \to (\Z_2)^{\cR(m,\ell+1)} \to \ol\sW_n
      \xrightarrow{\mu_n} \Z^{\cM(m,n)} \to 0
    \]
    where $\Ker\{\mu_n\}$ is identified with
    $(\Z_2)^{\cR(m,\ell+1)}\cong \Z_2 \otimes \sL_{\ell+1}$ via the
    higher order Sato-Levine invariant~$\SLev_{2\ell-1}$.
    Consequently,
    $\sW_n \cong \Z^{\cM(m,n)} \oplus (\Z_2)^{\cR(m,\ell+1)}$.
  \end{enumerate}
\end{theorem}

The \emph{higher-order Sato-Levine invariant $\SLev_{2k-1}$} which
appears in
Theorem~\ref{theorem:abelian-group-under-band-sum-framed}~(3) is
essential in this section.  Here we describe its definition following
\cite{Conant-Schneiderman-Teichner:2012-2}.  Recall that $\sD_n$ is
the kernel of the bracket map $\sL_1\otimes \sL_{n+1} \to \sL_{n+2}$
given by $X_i\otimes Y\mapsto [X_i,Y]$.  Suppose $n=2k$.  Due to
Levine~\cite[Theorem~1 and Corollary~2.2]{Levine:2001-2}, the quotient
of $\sD_{2k}$ modulo the subgroup generated by $\{\eta_{2k}(t) \mid t$
is an order $2k$ tree$\}$ is isomorphic to $\Z_2\otimes L_{k+1}$.  Let
$\slev_{2k}\colon \sD_{2k} \twoheadrightarrow \Z_2\otimes L_{k+1}$ be
the quotient map.

\begin{definition}[Higher order Sato-Levine invariant]
  \label{definition:higher-order-SL}
  For a link $L\subset S^3$ with $\mu_i(L)=0$ for $i\le 2k-1$,
  $\SLev_{2k-1}(L):=\slev_{2k}(\mu_{2k}(L))$.
\end{definition}

In Section~\ref{subsection:framed-cst-theory}, we will review some
necessary results on framed Whitney towers in $D^4$, from the work of
Conant, Schneiderman and Teichner.  In
Section~\ref{subsection:rational-whitney-filtration-framed}, we will
present a complete characterization of links in $\ol\W_n$ in terms of
the Milnor invariant and higher order Sato-Levine invariants (see
Theorem~\ref{theorem:rational-triviality-characterization-framed}),
and finally in Section~\ref{subsection:framed-graded-quotients}, we
will define the graded quotient $\ol\sW_n$ and compute its structure
to prove Theorem~\ref{theorem:structure-graded-quotient-framed}.

\subsection{Some results from the integral framed theory}
\label{subsection:framed-cst-theory}

All results discussed in this subsection are from the work Conant,
Schneiderman and Teichner~\cite{Conant-Schneiderman-Teichner:2012-2,
  Conant-Schneiderman-Teichner:2012-3,
  Conant-Schneiderman-Teichner:2012-1,
  Conant-Schneiderman-Teichner:2012-4}.  Similarly to the twisted
case, let $\W_n$ be the set of $m$-component framed links in $S^3$
bounding an order $n$ framed Whitney tower in~$D^4$.  Order $n+1$
framed Whitney tower concordance in $D^4$ is an equivalence relation
on~$\W_n$.  Let $\sW_n$ be the set of equivalence classes.  The band
sum of two classes is well defined on $\sW_n$ (particularly
independent of the choice of
bands)~\cite[Lemma~3.4]{Conant-Schneiderman-Teichner:2012-2}, and
$\sW_n$ is an abelian group under band sum.  Two links
$L, L' \in \W_n$ represent the same element in $\sW_n$ if and only if
$L\csumover{\beta}-L' \in \W_{n+1}$ for some~$\beta$.  Often we write
$\sW_n = \W_n / \W_{n+1}$.

In the study of the framed theory, they use a framed analog
$\widetilde\cT_n$ of the group $\cT^\tw_n$ discussed in
Section~\ref{section:review-integral-theory}.  The group
$\widetilde\cT_n$ is a quotient of the free abelian group generated by
order~$n$ trees (without using $\tw$-trees), modulo certain relations.
We omit its precise definition since we will use only the results
discussed below.  For an order $n$ framed Whitney tower $T$, the
formal sum $t^\tw_n(T)$ described in
Definition~\ref{definition:tree-invariant-of-twisted-tower} does not
have any $\tw$-tree summand and thus represents an element
$\tilde\tau_n(T) \in \widetilde\cT_n$.  Conversely, there is an
epimorphism $\widetilde R_n\colon \widetilde\cT_n \to \sW_n$, called
the \emph{realization map}, such that $\widetilde R_n(\phi)$ is the
class of a link bounding an order $n$ framed Whitney tower $T$ in
$D^4$ with $\tilde\tau_n(T)=\phi$.

\begin{theorem}[Framed Order
  Raising~{\cite[Theorem~4.4]{Conant-Schneiderman-Teichner:2012-2}}]
  \label{theorem:framed-order-raising}
  If a link $L$ bounds an order $n$ framed Whitney tower $T$ in $D^4$
  with $\tilde\tau_n(T)=0\in \widetilde\cT_n$, then $L$ bounds an
  order $n+1$ framed Whitney tower in~$D^4$.
\end{theorem}

For even $n$, they showed that
$\widetilde\cT_{2\ell} \cong \Z^{\cM(m,n)}$ where $m$ is the number of
link components, using their proof of the Levine
conjecture~\cite{Conant-Schneiderman-Teichner:2012-1,
  Conant-Schneiderman-Teichner:2012-2}.  (Recall that $\cM(m,n)$ is
the rank of~$\sD_n$; see Remark~\ref{remark:rank-of-D_n}.)  In fact,
there is a homomorphism $\widetilde\cT_{n} \to \cT^\tw_{n}$ taking the
class of an order $n$ tree to the class of the same tree, and for
$n=2\ell$, the composition
$\widetilde\cT_{2\ell} \to \cT^\tw_{2\ell} \xrightarrow{\eta_{2\ell}}
\sD_{2\ell}$ is a monomorphism whose image has the same rank
as~$\sD_{2\ell}$.

For odd $n=2\ell-1$, the structure of $\widetilde\cT_{2\ell-1}$ is
more involved, as described below.  The boundary twist operation
defined in \cite[Section~1.3]{Freedman-Quinn:1990-1} changes a twisted
Whitney disk to a framed Whitney disk at the cost of introducing new
intersections.  Using this (together with IHX), in
\cite{Conant-Schneiderman-Teichner:2012-2}, it was observed that a
twisted Whitney tower $T$ of order $2\ell$ can be changed to a framed
Whitney tower of order~$2\ell-1$, which we denote by~
$\partial^\tw(T)$.  In terms of the associated trees, this geometric
modification changes an order $\ell$ $\tw$-tree of the form
$\smallYtree{\tw}{i}{J}$ to the order $2\ell-1$
tree~$\smallYtree{i}{J}{J}\;$.  (Here $i$ denotes a univariant vertex
and $J$ is a subtree; any $\tw$-tree can be changed to the form of
$\smallYtree{\tw}{i}{J}$ by IHX.)  This gives rise to a homomorphism
$\partial^\tw\colon \cT^\tw_{2\ell} \to
\smash[t]{\widetilde\cT^{\mathstrut}_{2\ell-1}}$ satisfying
$\tilde\tau_{2\ell-1}(\partial^\tw(T))
= \partial^\tw(\tau^\tw_{2\ell}(T))$ for a twisted Whitney tower $T$
of order~$2\ell$.  The following commutative diagram, which we discuss
below, computes the structure
of~$\tilde\cT_{2\ell-1}$~\cite{Conant-Schneiderman-Teichner:2012-2,
  Conant-Schneiderman-Teichner:2012-3,
  Conant-Schneiderman-Teichner:2012-1,
  Conant-Schneiderman-Teichner:2012-4}.

\begin{equation}
  \label{equation:master}
  \vcenter{\hbox{
      \begin{tikzcd}[ampersand replacement=\&, row sep=scriptsize]
        \sB_{2\ell} \ar[d, tail] \ar[r, "\partial^\tw|", "\cong" below]
        \& \sB^{\SLev}_{2\ell-1}
        \smash[b]{\hbox to 0mm{ $\cong$ \small$\begin{cases}
            \Z_2\otimes \sL_k & \text{if }\ell=2k-1\\
            0 & \text{if }\ell=2k
          \end{cases}$\hss}}
        \ar[dd, tail] \&
        \\
        \cT^\tw_{2\ell} \ar[rr, "\partial^\tw" near start, crossing over]
        \ar[rd, two heads] \ar[dd, "\eta_{2\ell}", two heads]
        \& \& \widetilde\cT_{2\ell-1} \ar[r, two heads]
        \& \cT^\tw_{2\ell-1} \ar[d, "\cong" left, "\eta_{2\ell-1}"]
        \\
        \& \Z_2 \otimes \sL'_{\ell+1} \ar[d,two heads] \ar[ru, tail]
        \& \& \sD_{2\ell-1}
        \\
        \sD_{2\ell} \ar[r, "\slev_{2\ell}" below, two heads]
        \& \Z_2 \otimes \sL_{\ell+1}        
      \end{tikzcd}
    }}
\end{equation}

\begin{enumerate}
\item The row starting with $\cT^\tw_{2\ell}$ is exact.  By
  Theorem~\ref{theorem:cst-integral-theory}~(\ref{item:kernel-of-mu}),
  $\cT^\tw_{2\ell-1} \cong \sD_{2\ell-1}$ under~$\eta_{2\ell-1}$.
\item The image of
  $\partial^\tw\colon \cT^\tw_{2\ell} \to
  \widetilde\cT^{\mathstrut}_{2\ell-1}$ is isomorphic to
  $\Z_2\otimes\sL'_{\ell+1}$, where $\sL'_{\ell+1}$ is the degree
  $\ell+1$ part of Levine's \emph{quasi-Lie
    algebra}~\cite{Levine:2001-2}.  The abelian group $\sL'_{\ell+1}$
  is defined by replacing the alternativity relation $[X,X]=0$ of
  $\sL_{\ell+1}$ with the antisymmetry relation $[X,Y]+[Y,X]=0$.
  Regarding $\sL'_{\ell+1}$, we will need only (3) and (4) below.
\item There is a homomorphism
  $\Z_2\otimes\sL'_{\ell+1} \to \Z_2\otimes\sL_{\ell+1}$ taking the
  class of an $(\ell+1)$-fold bracket in $\sL'_{\ell+1}$ to the class
  of the same bracket in~$\sL_{\ell+1}$.  It is an epimorphism fitting
  into the bottom left square.
\item Let
  $\sB^{\SLev}_{2\ell-1} := \Ker\{\Z_2\otimes\sL'_{\ell+1} \to
  \Z_2\otimes\sL_{\ell+1}\}$.  Then $\partial^\tw$ restricts to an
  isomorphsm $\sB_{2\ell} \cong \sB^{\SLev}_{2\ell-1}$.  Recall that
  $\sB_n$ is the kernel of $\eta_{n}\colon \cT^\tw_n \to \sD_n$, and
  isomorphic to $\Z_2\otimes \sL_k$ if $n=4k-2$, and $0$ if
  $n\not\equiv 2 \mod4$, as discussed in
  Section~\ref{section:review-integral-theory}.
\end{enumerate}

\subsection{Framed rational Whitney tower filtration}
\label{subsection:rational-whitney-filtration-framed}

\begin{lemma}
  \label{lemma:kernel-realization-by-rationally-slice-links-framed}
  For any $\phi\in \sB^{\SLev}_{4k-3} \cong \Z_2\otimes\sL_k$, there
  is a rationally slice link $L(\phi)$ in $S^3$ which bounds a framed
  Whitney tower $T$ of order $4k-3$ in $D^4$ with
  $\tilde\tau_{4k-3}(T)=\phi$.
\end{lemma}

\begin{proof}
  By Lemma~\ref{lemma:kernel-realization-by-rationally-slice-links},
  there is a rationally slice link $L$ which bounds a twisted Whitney
  tower $T'$ of order $4k-2$ in $D^4$ with
  $\tau^\tw_{4k-2}(T') = (\partial^\tw)^{-1}(\phi)$.  Then
  $T:=\partial^\tw(T')$ is an order $4k-3$ framed tower in $D^4$
  bounded by $L$, and
  $\tilde\tau_{4k-3}(T) = \partial^\tw(\tau^\tw_{4k-2}(T')) = \phi$.
  So $L(\phi):=L$ satisfies the desired properties.
\end{proof}

Similarly to the convention for $\sB_n$ in the twisted case, let
$\sB^{\SLev}_{2\ell}=0$ for brevity.  Then
Lemma~\ref{lemma:kernel-realization-by-rationally-slice-links-framed}
holds for any order $n$ as well as~$n=4k-3$.

The following is a framed case analog of
Lemma~\ref{lemma:rational-vs-integral-towers-for-link}.

\begin{lemma}
  \label{lemma:rational-vs-integral-towers-for-link-framed}
  Suppose $L$ is a link bounding a framed order $n$ Whitney tower
  in~$D^4$.  Then the following are equivalent:
  \begin{enumerate}
  \item There is $\phi\in \sB^{\SLev}_{n}$ such that any band sum
    $L\csumover{\beta}L(\phi)$ bounds a framed order $n+1$ Whitney
    tower in~$D^4$.
  \item $L$ bounds a framed Whitney tower of order $n+1$ in a
    $\Z[\frac12]$-homology 4-ball.
  \item $L$ bounds a framed Whitney tower of order $n+1$ in a rational
    homology 4-ball.
  \item $\mu_n(L)=0$, and in addition when $n=2\ell-1$,
    $\SLev_{2\ell-1}(L)=0$.
  \end{enumerate}
\end{lemma}

\begin{proof}
  (1)~$\Rightarrow$~(2) is proven in the exactly same way as
  (1)~$\Rightarrow$~(2) of
  Lemma~\ref{lemma:kernel-realization-by-rationally-slice-links},
  using that $L(\phi)$ is rationally slice.  (2)~$\Rightarrow$~(3) is
  trivial.

  Suppose (3) holds, that is, there is an order $n+1$ framed Whitney
  tower $T$ in a rational homology 4-ball bounded by~$L$.  Since $T$
  is an order $n+1$ twisted Whitney tower, $\mu_n(L) = 0$ by
  Theorem~\ref{theorem:whitney-tower-milnor-invariant}.  If
  $n=2\ell-1$, then
  $\SLev_{2\ell-1}(L) = \slev_{2\ell}(\mu_{2\ell}(L)) =
  \slev_{2\ell}(\eta_{2\ell}(t^\tw_{2\ell}(T)))$ by
  Definition~\ref{definition:higher-order-SL} and
  Theorem~\ref{theorem:whitney-tower-milnor-invariant}.  Since $T$ is
  framed, $t^\tw_{2\ell}(T)$ has no $\tw$-tree summand, that is, all
  the summands are order $2\ell$ trees.  By the definition of
  $\slev_{2\ell}$, it follows that
  $\slev_{2\ell}(\eta_{2\ell}(t^\tw_{2\ell}(T)))=0$.  This shows that
  (4) holds.

  Suppose (4) holds.  If $n=2\ell$, then for any fixed order $n$
  framed Whitney tower $T$ in $D^4$ bounded by $L$,
  $\eta_{2\ell}(\tilde\tau_{2\ell}(T)) = \mu_{2\ell}(L) = 0$.  Since
  $\widetilde\cT_{2\ell} \to \cT^\tw_{2\ell}
  \xrightarrow{\eta_{2\ell}} D_{2\ell}$ is injective,
  $\tilde\tau_{2\ell}(T)=0$ in~$\widetilde\cT_{2\ell}$.  By
  Theorem~\ref{theorem:framed-order-raising}, $L$ bounds a framed
  order $n+1$ Whitney tower in~$D^4$.  In particular, (1) holds (with
  $\phi=0$).  If $n=2\ell-1$, then since $\mu_{2\ell-1}(L)=0$ and
  $\mu_{2\ell-1}\colon \sW^\tw_{2\ell-1} \to D_{2\ell-1}$ is an
  isomorphism by
  Theorem~\ref{theorem:cst-integral-theory}~(\ref{item:kernel-of-mu}),
  $L$ bounds an order $2\ell$ twisted Whitney tower $T$ in~$D^4$.
  Using the hypothesis and
  Theorem~\ref{theorem:whitney-tower-milnor-invariant}, we obtain
  \[
    0=\SLev_{2\ell-1}(L) = \slev_{2\ell}(\mu_{2\ell}(L)) =
    \slev_{2\ell}(\eta_{2\ell}(\tau^\tw_{2\ell}(T))).
  \]
  It follows that
  $\tilde\tau_{2\ell-1}(\partial^\tw(T))
  = \partial^\tw(\tau^\tw_{2\ell}(T)) \in \sB^{\SLev}_{2\ell-1}$,
  using the diagram~\eqref{equation:master}.  Let
  $\phi = -\tilde\tau_{2\ell-1}(\partial^\tw(T))$.  Then any band sum
  $L \csumover{\beta} L(\phi)$ bounds a framed Whitney tower $T'$ with
  $\tilde\tau_{2\ell-1}(T') = \tilde\tau_{2\ell-1}(\partial^\tw(T)) +
  \phi = 0 \in \widetilde\cT_{2\ell-1}$.  By
  Theorem~\ref{theorem:framed-order-raising},
  $L \csumover{\beta} L(\phi)$ bounds a framed order $n+1$ Whitney
  tower in $D^4$.  This completes the proof of (4)~$\Rightarrow$~(1).
\end{proof}

Once
Lemma~\ref{lemma:kernel-realization-by-rationally-slice-links-framed}
is given, the following theorem is proven by the argument of the proof
of its twisted analog
Theorem~\ref{theorem:rational-triviality-characterization}, using
Lemma~\ref{lemma:kernel-realization-by-rationally-slice-links-framed}
in place of
Lemma~\ref{lemma:kernel-realization-by-rationally-slice-links}.

\begin{theorem}
  \label{theorem:rational-triviality-characterization-framed}
  For a link $L$ in $S^3$ and $n\ge 0$, the following are equivalent:
  \begin{enumerate}
  \item $L\in \ol\W_{n+1}$, that is, $L$ bounds a framed
    Whitney tower of order $n+1$ in a rational homology 4-ball.
  \item $\mu_k(L)=0$ for $k\le n$, and in addition when $n=2\ell-1$,
    $\SLev_{2\ell-1}(L) = 0$.
  \item For any basing $b$ for $L$, there is a rationally slice link
    $L_0$ with a basing $b_0$ such that
    $L\csumover{(b,b_0)}L_0 \in \W_{n+1}$.
  \item For any basing $b$ for $L$, there is a $\Z[\frac12]$-slice
    link $L_0$ with a basing $b_0$ such that
    $L\csumover{(b,b_0)}L_0 \in \W_{n+1}$.
  \item $L$ bounds a twisted Whitney tower of order $n+1$ in a
    $\Z[\frac12]$-homology 4-ball.
  \end{enumerate}
\end{theorem}

We omit details of the proof.

\subsection{Group structure on the rational framed graded quotients}
\label{subsection:framed-graded-quotients}

In this subsection we will formulate the ``graded quotient''
$\ol\sW_n$ of the rational framed filtration~$\{\ol\W_n\}$ and compute
its structure.  Rather unexpectedly, the main remaining difficulty is
to show that the graded quotient has a group structure under band sum.
Once this is resolved, the group can be computed via Milnor invariants
and higher order Sato-Levine invariants, using
Theorem~\ref{theorem:rational-triviality-characterization-framed}.  To
estabilish a group structure, it appears to have significant advantage
to adapt the following definition of an equivalence relation, instead
of framed Whitney tower concordance.

\begin{definition}
  \label{definition:framed-equivalence-relation}
  On the set $\ol\W_n$, define a relation $\approx$ by
  $L\approx L'$ if $L\csumover{\beta}-L' \in \ol\W_{n+1}$ for
  some~$\beta$.
\end{definition}

\begin{lemma}
  \label{lemma:framed-equivalence-relation}
  On $\ol\W_n$, $\approx$ is an equivalence relation.
\end{lemma}

It is straightforward to verify that $\approx$ is symmetric and
reflexive.  In the proof of transitivity, we use the following two
facts: (i) a link in $\ol\W_n$ can always be represented by a link in
$\W_n$, due to
Theorem~\ref{theorem:rational-triviality-characterization-framed}, and
(ii) band sum is well-defined on $\sW_n = \W_n/\W_{n+1}$, due
to~\cite{Conant-Schneiderman-Teichner:2012-2}.

\begin{proof}[Proof of Lemma~\ref{lemma:framed-equivalence-relation}]
  
  We will prove transitivity.  Suppose $L$, $L'$ and $L''$ are in
  $\ol\W_n$ and $L\csumover{\beta}-L'$, $L'\csumover{\gamma}-L''$ are
  in~$\ol\W_{n+1}$.  We need to show that $L\csumover{\alpha}-L''$ is
  in $\ol\W_{n+1}$ for some choice of bands~$\alpha$.

  In what follows, we will repeatedly use a standard fact that if
  $L_0$ is rationally slice, then for any link $L$ and for any
  $\beta$, $L\csumover{\beta}L_0$ is rationally concordant to~$L$.
  The proof is straightforward: choose slice disks $\Delta$ for $L_0$
  in a rational homology $S^3\times I$, choose an arc in the rational
  homology $S^3\times I$ which joins two boundary components and which
  is disjoint from $\Delta$, and replace a tubular neighborhood of the
  arc with $(S^3\sm(\text{3-ball disjoint from $L$}),L)\times I$ to
  obtain a cobordism between $L\sqcup L_0$ and~$L$.  Attach to this a
  standard genus zero cobordism in $S^3\times I$ between
  $L\csumover{\beta}L_0$ and $L\sqcup L_0$ to obtain a concordance
  between $L\csumover{\beta}L_0$ to $L$ in a rational homology
  $S^3\times I$.  The same argument shows that if $L_0\in \ol\W_n$,
  then $L\csumover{\beta}L_0$ is order $n$ framed Whitney tower
  concordant to~$L$ in a rational homology~$S^3\times I$.
  
  Begin with the split union $L\sqcup -L' \sqcup L' \sqcup -L''$, and
  regard $\beta$ and $\gamma$ as disjoint bands joining components of
  sublinks of this split union.  Choose a collection of bands $\delta$
  disjoint from $\beta$ and $\gamma$ to define a band sum
  $-L'\csumover\delta L'$ of the sublinks $-L'$ and~$L'$.  Then
  \[
    J := (L\csumover{\beta}-L')\csumover{\delta}(L'\csumover{\gamma}-L'')
  \]
  is defined.  The link $J$ bounds a framed Whitney tower of order
  $n+1$ in a rational homology 4-ball, since so do
  $L\csumover{\beta}-L'$ and $L'\csumover{\gamma}-L''$.  Fix
  arbitrarily given bands $\alpha$ on $L\sqcup -L''$ to define a band
  sum $L\csumover{\alpha}-L''$.  We claim that $J$ is order $n+1$
  framed Whitney tower concordant to $L\csumover{\alpha}-L''$ in some
  rational homology $S^3\times I$.  Stacking the claimed Whitney tower
  concordance with the above Whitney tower bounded by $J$, it follows
  that $L\csumover{\alpha}-L''$ bounds a framed Whitney tower of
  order~$n+1$ in a rational homology 4-ball.  This completes the
  proof.

  The remaining part is devoted to the proof of the claim.  For a
  basing $c$ of a link $R$, the mirror image of $c$ with reversed
  orientation is a basing of~$-R$.  Denote this basing by~$-c$.  Any
  basing of $-R$ is of the form of~$-c$.  Choose basings $b$, $b'$ and
  $-b''$ for the sublinks $L$, $L'$ and $-L''$ of the split union
  $L \sqcup -L' \sqcup L' \sqcup -L''$ respectively.  We may assume
  that $b$, $b'$ and $-b''$ are mutually disjoint and disjoint from
  $\beta$, $\gamma$ and~$\delta$.  Also, we may assume that $-b'$, as
  a basing of the sublink $-L'$ of the split union, is disjoint from
  all other basings and bands.  Invoke
  Lemma~\ref{lemma:rational-vs-integral-towers-for-link-framed} to
  choose rationally slice links $L_0$, $L'_0$ and $L''_0$ with basings
  $b_0$, $b'_0$, $b''_0$ such that the connected sums
  $L\csumover{(b,b_0)}L_0$, $L'\csumover{(b',b'_0)}L'_0$ and
  $L''\csumover{(b'',b''_0)}L''_0$ are in~$\W_n$.  Let
  \[
    J' := \bigg(\Big(\big((J \csumover{(b,b_0)} L_0)
    \csumover{(-b',-b'_0)} -L'_0 \big) \csumover{(b',b'_0)} L'_0 \Big)
    \csumover{(-b'',-b''_0)} -L''_0\bigg).
  \]
  Since $L_0$, $L'_0$ and $L''_0$ are rationally slice, $J'$ is
  rationally concordant to~$J$.  Define
  $L_1 := L\csumover{(b,b_0)}L_0$,
  $L'_1 := L'\csumover{(b',b'_0)}L'_0$ and
  $L''_1 := L''\csumover{(b'',b''_0)}L''_0$.  Since our bands and
  basings are mutually disjoint, all the band sum and connect sum
  operations are associative.  In particular, we have
  \[
    J' = L_1 \csumover{\beta} -L_1' \csumover{\delta} L'_1
    \csumover{\gamma} -L''_1.
  \]
  Choose a basing $c$ for~$L'_1$ to define a connected sum
  $-L'_1 \csumover{(-c,c)} L'_1$.  Recall that $\alpha$ is the bands
  on $L\sqcup -L''$ given above.  We may assume that both $b$ for $L$
  and $-b''$ for $-L''$ have been chosen to be disjoint from~$\alpha$.
  Then, using $\alpha$ as bands on $L_1\sqcup -L''_1$, a band sum
  $L_1\csumover{\alpha}-L''_1$ is defined.  Choose a collection of
  bands $\epsilon$ to define
  \[
    J'' := (-L'_1 \csumover{(-c,c)} L'_1) \csumover{\epsilon}
    (L_1\csumover{\alpha}-L''_1).
  \]
  Since $L_1$, $L_1'$, $-L_1'$ and $L_1''$ are in $\W_n$, a band sum
  of them is well-defined in $\sW_n = \W_n/\W_{n+1}$, independent of
  the choice of the bands.  Therefore, $J'$ and $J''$ are order $n+1$
  framed Whitney tower concordant in $S^3\times I$.  Since the
  connected sum $-L'_1 \csumover{(-c,c)} L'_1$ is slice in $D^4$,
  $J''$ is concordant to $L_1\csumover{\alpha}-L''_1$.  Since
  \[
    L_1\csumover{\alpha}-L''_1 = L_0 \csumover{(b_0,b)}
    (L\csumover{\alpha}-L'') \csumover{(-b'',-b''_0)} -L''_0,
  \]
  and since $L_0$ and $L''_0$ are rationally slice,
  $L_1\csumover{\alpha}-L''_1$ is rationally concordant to
  $L\csumover{\alpha}-L''$.  This completes the proof of the claim
  that $L\csumover{\alpha}-L''$ is order $n+1$ framed Whitney tower
  concordant to $J$.
\end{proof}

Let $\ol\sW_n$ be the set of equivalence classes of links in $\ol\W_n$
under~$\approx$.  Denote by $[L]\in \ol\sW_n$ the equivalence class of
a link~$L\in \ol\W_n$.

\begin{theorem}
  \label{theorem:abelian-group-under-band-sum-framed}
  The band sum operation $[L]+[L'] := [L\csumover{\beta}L']$ is
  well-defined on $\ol\sW_n$, and $\ol\sW_n$ is an abelian
  group under the band sum operation.
\end{theorem}

\begin{proof}
  Once we show that the band sum operation is well-defined, it follows
  immediately that $\ol\sW_n$ is an abelian group; the identity is the
  class of a trivial link, and the inverse of $[L]$ is $[-L]$, the
  class of the mirror image of $L$ with reversed orientation, since
  $L\csumover{(b,-b)}-L$ is slice.

  In what follows we will prove the well-definedness.  Suppose
  $P \approx Q$ and $P' \approx Q'$ in $\ol\W_n$, that is,
  $P\csumover\alpha -Q$ and $P'\csumover{\alpha'} -Q'$ are
  in~$\ol\W_{n+1}$ for some $\alpha$ and~$\alpha'$.  We need to
  show that $P\csumover\beta P' \approx Q\csumover\gamma Q'$ for any
  given $\beta$ and~$\gamma$.  We will proceed using essentially the
  same technique as that of the proof of
  Lemma~\ref{lemma:framed-equivalence-relation}.

  Regard $P$, $-Q$, $P'$ and $-Q'$ as sublinks of
  $P\sqcup -Q \sqcup P' \sqcup -Q'$, and choose $\delta$ disjoint from
  $\alpha$ and $\alpha'$ to define~$P\csumover\delta P'$.  Then
  \[
    J := (P\csumover\alpha -Q) \csumover\delta
    (P'\csumover{\alpha'}-Q')
  \]
  lies in~$\ol\W_{n+1}$ since both $P\csumover\alpha-Q$ and
  $P'\csumover{\alpha'}-Q'$ are in~$\ol\W_{n+1}$.
  
  Choose basings $b$, $-c$, $b'$ and $-c'$ of $P$, $-Q$, $P'$ and
  $-Q'$ respectively, in such a way that they are mutually disjoint
  and are disjoint from $\alpha$, $\alpha'$, $\beta$, $\gamma$
  and~$\delta$.  Appealing to
  Theorem~\ref{theorem:rational-triviality-characterization}, choose
  rationally slice links $P_0$, $Q_0$, $P'_0$ and $Q'_0$ with basings
  $b_0$, $c_0$, $b'_0$ and $c'_0$ such that
  $P_1 := P\csumover{(b,b_0)}P_0$, $Q_1 := Q\csumover{(c,c_0)}Q_0$,
  $P'_1 := P'\csumover{(b',b'_0)}P'_0$ and
  $Q'_1 := Q'\csumover{(c',c'_0)}Q'_0$ are in~$\W_n$.  Then
  \[
    J' := (P_1\csumover\alpha -Q_1) \csumover\delta
    (P'_1\csumover{\alpha'}-Q'_1)
  \]
  is rationally concordant to~$J$.

  Choose $\epsilon$ disjoint from $b$, $b'$, $-c$, $-c'$, $\beta$ and
  $\gamma$ to define $P\csumover\epsilon -Q$.  Then
  \[
    J'' := (P_1\csumover\beta P'_1) \csumover\epsilon
    -(Q_1\csumover\gamma Q'_1)
  \]
  is defined.  Furthermore, since band sum is well-defined on
  $\sW_n = \W_n/\W_{n+1}$ independent of the choice of bands and since
  $P_1$, $Q_1$, $P'_1$, $Q'_1 \in \W_n$, $J''$ is order $n+1$ framed
  Whitney tower concordant, in $S^3\times I$, to~$J'$.  Since $P_0$,
  $P'_0$, $Q_0$, $Q'_0$ are rationally slice,
  \[
    J''' := (P\csumover\beta P') \csumover\epsilon -(Q\csumover\gamma Q')
  \]
  is rationally concordant to~$J''$.  Combining the above, it follows
  that $J''' \in \ol\W_{n+1}$, that is,
  $P\csumover\beta P' \approx Q\csumover\gamma Q'$.
\end{proof}

Now we compute the structure of~$\ol\sW_n$.  Recall that $\cM(m,n)$ is
the number of linearly independent Milnor invariants of order $n$ (see
Remark~\ref{remark:rank-of-D_n}).


\begin{proof}[Proof of Theorem~\ref{theorem:structure-graded-quotient-framed}]
  Since $\mu_n(L)$ vanishes for links $L$ in $\ol\W_{n+1}$ by
  Theorem~\ref{theorem:rational-triviality-characterization-framed}
  and since $\mu_n$ is additive under band sum, $\mu_n$ induces a
  group homomorphism $\ol\sW_n\to\sD_n$.  Recall that any class
  $[L]\in \ol\sW_{n}$ is represented by a link $L\in \W_n$ by
  Theorem~\ref{theorem:rational-triviality-characterization-framed}.
  It follows that the natural map $\W_n\to \ol\W_n$ and the induced
  homomorphism $\sW_n\to \ol\sW_n$ are surjective.  Since
  \[
    \begin{tikzcd}
      \sW_n \ar[rr, two heads] \ar[rd, "\mu_n" below left]
      & & \ol\sW_n \ar[ld, "\mu_n"]
      \\
      & \sD_n
    \end{tikzcd}
  \]
  is commutative, the image $\mu_n(\ol\sW_n) \subset \sD_n$ is equal
  to $\mu_n(\sW_n)$.  It is known that $\mu_n(\sW_n)$ has the same
  rank as $\sD_n$, namely has rank~$\cM(m,n)$; for, since the
  realization $\widetilde R_n \colon \widetilde\cT_n \to \sW_n$ is
  surjective and $\mu_n(L)=\eta_n(t^\tw_n(T))$ for a bounding order
  $n$ framed Whitney tower $T\subset D^4$ by
  \cite[Theorem~6]{Conant-Schneiderman-Teichner:2012-3} or
  Theorem~\ref{theorem:whitney-tower-milnor-invariant}, $\mu_n(\sW_n)$
  is equal to the image of $\widetilde\cT_n \to \cT^\tw_n \to \sD_n$,
  which has rank $\cM(m,n)$ as stated in
  Section~\ref{subsection:framed-cst-theory}.  This shows
  Theorem~\ref{theorem:structure-graded-quotient-framed}~(1).

  For $n=2\ell$, $[L]=0$ in $\ol\sW_n$ if and only if $\mu_n(L)=0$,
  by
  Theorem~\ref{theorem:rational-triviality-characterization-framed}.
  From this Theorem~\ref{theorem:structure-graded-quotient-framed}~(2)
  follows.

  For $n=2\ell-1$, $\SLev_{2\ell-1}=\slev_{2\ell}\circ\mu_{2\ell}$ on
  $\Ker\{\mu_{2\ell-1}\}$ is additive under band sum since so
  is~$\mu_{2\ell}$.  Therefore there is an induced homomorphism
  $\SLev_{2\ell-1}\colon \Ker\{\mu_{2\ell-1}\} \to
  \Z_2\otimes\sL_{\ell+1}$.  This is an epimorphism.  For,
  $\mu_{2\ell}\colon \ol\sW_{2\ell}^\tw \to \sD_{2\ell}$ is an
  isomorphism by
  Theorem~\ref{theorem:computation-of-twisted-graded-quotient}~(2),
  and consequently the composition
  $\slev_{2\ell}\circ \mu_{2\ell} \colon \ol\sW_{2\ell}^\tw \to
  \sD_{2\ell} \to \Z_2\otimes\sL_{\ell+1}$ with the quotient
  homomorphism $\slev_{2\ell}$ is surjective.  Since every
  $L\in \ol\W^\tw_{2\ell}$ represents an element
  $[L] \in \Ker\{\mu_{2\ell-1}\} \subset \ol\sW_{2\ell-1}$ and
  $\SLev_{2\ell-1}(L)=\slev_{2\ell}(\mu_{2\ell}(L))$ by the definition
  of~$\SLev_{2\ell-1}$, it follows that
  $\SLev_{2\ell-1}\colon \Ker\{\mu_{2\ell-1}\} \to
  \Z_2\otimes\sL_{\ell+1}$ is surjective.

  If $\mu_{2\ell-1}(L)=0$, then $[L]=0$ in $\ol\sW_n$ if and only if
  $\SLev_{2\ell-1}(L)=0$ by
  Theorem~\ref{theorem:rational-triviality-characterization-framed}.
  It follows that
  $\SLev_{2\ell-1}\colon \Ker\{\mu_{2\ell-1}\} \to
  \Z_2\otimes\sL_{\ell+1}$ is injective.  This completes the proof of
  Theorem~\ref{theorem:structure-graded-quotient-framed}~(3).
\end{proof}


\bibliographystyle{amsalpha}
\def\MR#1{}
\bibliography{research}

\end{document}